\documentclass{article}
\usepackage{geometry} 
\usepackage{graphicx}
\usepackage{amsmath}
\usepackage{latexsym}
\usepackage{amssymb}
\usepackage{appendix}

\newtheorem{theorem}{Theorem}
\newtheorem{proposition}[theorem]{Proposition}%
\newtheorem{corollary}[theorem]{Corollary}%
\newtheorem{lemma}[theorem]{Lemma}%
\newtheorem{remark}{Remark}%
\newtheorem{proof}{Proof}%
\newtheorem{definition}{Definition}%

\newcommand{\bmeta}{\bm{\eta}}
\newcommand{\bmmu}{\bm{\mu}}

\newcommand{\bm}[1]{\mbox{\boldmath$#1$}}

\newcommand{\bmV}{\bm{V}}
\newcommand{\bmb}{\bm{b}}

\newcommand{\bma}{\bm{a}}

\newcommand{\bmu}{\bm{u}}
\newcommand{\bmv}{\bm{v}}
\newcommand{\bmw}{\bm{w}}

\newcommand{\bmr}{\bm{r}}
\newcommand{\bmq}{\bm{q}}

\newcommand{\bmzero}{\bm{0}}
\newcommand{\bphi}{\bm{\phi}}

\newcommand{\bpsi}{\bm{\psi}}
\newcommand{\bxi}{\bm{\xi}}

\newcommand{\lng}{\left\langle  \right.}
\newcommand{\rng}{\left.  \right\rangle}

\newcommand{\Real}{\mathop{\mathrm{Re}}\nolimits}
\newcommand{\Imag}{\mathop{\mathrm{Im}}\nolimits}

	\title{Arbitrarily weak head-on collision can induce annihilation \\--The role of hidden instabilities--}
	
	\author{Yasumasa Nishiura\thanks{Research Center of Mathematics for Social Creativity,
			Research Institute for Electronic Science, Hokkaido University}
			\thanks{Chubu University Academy of Emerging Sciences, Chubu University}, 
			Takashi Teramoto\thanks{Faculty of Data Science, Kyoto Women's University}, 
			Kei-Ichi Ueda\thanks{Faculty of Science, Academic Assembly, University of Toyama}}
			
	\date{}
\begin{document}
	\maketitle
\begin{abstract}
		In this paper, we focus on annihilation dynamics for the head-on collision of traveling patterns. A representative and well-known example of annihilation is the one observed for 1-dimensional traveling pulses of the FitzHugh-Nagumo equations. In this paper, we present a new and completely different type of annihilation arising in a class of three-component reaction diffusion system. It is even counterintuitive in the sense that the two traveling spots or pulses come together very slowly but do not merge, keeping some separation, and then they start to repel each other for a certain time. Finally, up and down oscillatory instability emerges and grows enough for patterns to become extinct eventually (see Figs. 1-3). There is a kind of hidden instability embedded in the traveling patterns, which causes the above annihilation dynamics. The hidden instability here turns out to be a codimension 2 singularity consisting of drift and Hopf (DH) instabilities, and there is a parameter regime emanating from the codimension 2 point in which a new type of annihilation is observed. The above scenario can be proved analytically up to the onset of annihilation by reducing it to a finite-dimensional system. Transition from preservation to annihilation is also discussed in this framework.
\end{abstract}
	
	
	\section{Introduction}
	A variety of spatially localized structures emerge in many 
	fields, such as gas discharge systems 
	\cite{astrov-purwins,bode-liehr,purwins-astrov,schenk-orguil}, binary convection \cite{kolodner_1991,watanabe_jfm}, CO oxidation \cite{bar-eiswirth,zimmermann-firle}, desertification \cite{vegetation_rietkerk_2001,vegetation_tlide_2015}, oscillon \cite{swinney_oscillon}, morphogenesis \cite{yochelis_2008,champ_2014,cell_2015},
	and chemical reactions\cite{kepper-perraud,lee-mccormick,epstein_chaos_review}. They may be stationary, time-periodic, or traveling patterns as coherent objects.
	It is noteworthy that, unlike elastic balls, those localized patterns can behave like a living cell, i.e., replication, coalescence, and annihilation according to environmental changes and external perturbations, such as collisions with other patterns or inhomogeneities in the media \cite{arge2,nishiura-teramoto1,nishiura-teramoto2,teramoto-ueda-nishiura,nishiura-chaos15}. 
	In other words, they are stable when isolated, but the hidden instabilities come out via interactions with the external world. To elucidate the mathematical mechanism behind these dynamics is a difficult task because they contain large deformations and strong nonlinear interactions. Nevertheless, by numerical exploration, we can decompose the process into the sequence of deformations forming a network of saddle solutions with decreasing Morse index, called the {\it scattor-network} developed by 
	\cite{nishiura-teramoto1,nishiura-teramoto2}, which gives us a good perspective of complex transient dynamics. For instance, suppose that two traveling spots collide at some point; then, whether they merge or not is controlled by a saddle type of stationary solution with a peanut shape called the {\it scattor}, in the sense that the solution orbit can be sorted out by the stable manifold of this scattor. See \cite{nishiura-teramoto1,nishiura-teramoto2,nishiura-chaos15} for details and other examples of scattor-network.
	
	However, it is possible to construct a rigorous method when we focus on a {\it weak} interaction among localized patterns, for instance, Ei et al. \cite{ei,ei-mimura-nagayama} proposed a center manifold reduction method, in which the dynamics of pulse solutions can be described by low-dimensional ODEs as far as
	the solutions are well-separated and interact weakly (see also \cite{doel-2003,ohta,scheel-wright,wright-2010}). 
	
	These two complementary approaches have enabled us to clarify the onset mechanism of the complex dynamics of localized patterns, such as self-replication, self-destruction, and oscillation of front and pulse solutions in one dimension \cite{nishiura-ueyama1,nishiura-ueyama2,ei-nishiura-ueda,kolokolnikov,kolokolnikov2,ei2}, pulse generators \cite{nishiura-teramoto8},
	and rotational motion of spots in two dimensions \cite{nishiura-teramoto6} and others (see references therein). 
	
	In this paper, we are particularly interested in annihilation dynamics, which is a typical phenomenon when two localized patterns collide strongly, as is well-known for the FitzHugh-Nagumo (FHN) pulses and 
	BZ patterns \cite{arge,arge2,keener-sneyd}, though the rigorous
	verification still remains open. An exceptional case is a class of scalar reaction diffusion equation in which one can control the large deformation of front waves, as shown in \cite{fukao-morita-2004,yagisita-2003}. The key ingredient here is an argument of comparison type; however, it cannot be applied to the systems discussed here.

	The manner of annihilation of FHN pulses is intuitively clear because of the concentration of inhibitor contaminates around the colliding point, which leads to the collapse of the peak of the activator. We call this {\it annihilation of fusion type} (AF) for later use because the colliding patterns merge into one body before extinction. Annihilation of AF type can be understood by using the framework scattor-network mentioned above for traveling pulses of the Gray-Scott model and traveling spots of the three-component systems (\ref{eqn:gd}) and (\ref{gs3-pde}) (see \cite{nishiura-teramoto1,nishiura-teramoto2,nishiura-chaos15} for details).

	However, this is not a unique way that annihilation occurs at collision; in fact, a complete different type of annihilation of traveling spots has been found for the following three-component system of the FitzHugh-Nagumo type (\ref{eqn:gd}), as shown in Fig.\ref{fig:fhn3}.
	
	\begin{equation}\label{eqn:gd}
		\left\{
		\begin{array}{ll}
			u_t =& \displaystyle{ D_u \triangle u  +k_2 u - u^3 -k_3 v - k_4 w + k_1 },
			\vspace*{2mm} \\
			\tau v_t = & \displaystyle{ D_v  \triangle v + u - \gamma v},
			\vspace*{2mm} \\
			\theta w_t = & \displaystyle{ D_w  \triangle w + u -w },
		\end{array}
		\right.
	\end{equation}
	where $k_{i}, \tau, \theta, \gamma$ as well as the diffusion coefficients $D_u, D_v, D_w$ are all positive constants. The system (\ref{eqn:gd}) was introduced by Purwins et al.  \cite{astrov-purwins,bode-liehr,purwins-astrov,schenk-orguil} as a qualitative model of gas-discharge phenomena, and many interesting dynamics of localized patterns have been demonstrated both in numerical and analytical approaches \cite{nishiura-teramoto1,nishiura-teramoto2,nishiura-teramoto3,nishiura-teramoto4,nishiura-teramoto5,nishiura-teramoto7,nishiura-teramoto8}, especially in the excitable regime. 

	\begin{figure}[htbp]
			\centering
		\includegraphics[width=11cm,clip]{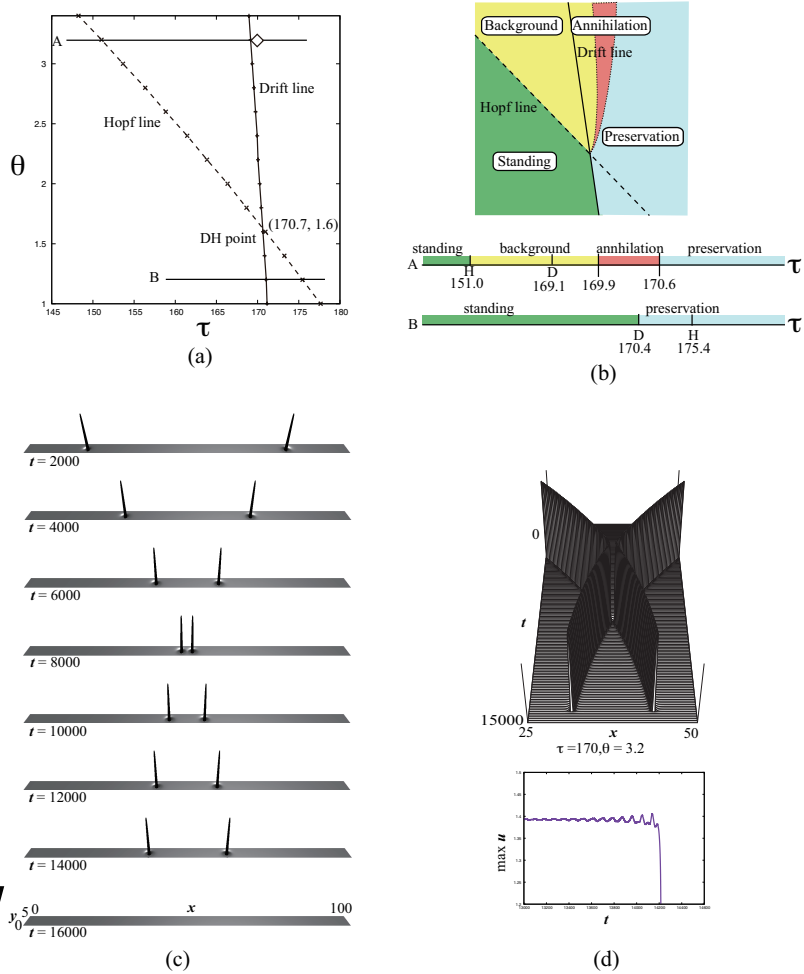}
		\caption{ (a) (b) Phase diagram for the outputs of two colliding traveling spots in \eqref{eqn:gd} with respect to $(\tau, \theta)$ and its schematic. The broken (respectively, dotted) line indicates Hopf (respectively, drift) bifurcation for the standing spot. A codimension 2 point of DH type appears at the intersecting point $(\tau, \theta) = (170, 3.2)$. Along the two lines A and B, the outputs of colliding spots are shown as $\tau$ is increased. The annihilation region (red) is inflated for clarity. 
			(c) Annihilation behavior at the diamond (white) point in the phase diagram. 
			Only profiles of the $u$-component are shown. They interact weakly and rebound while keeping their shape; however, they start to oscillate and disappear eventually.
			(d) A magnified birds-eye view plot near annihilation point. Up and down oscillation grows just before the annihilation.
			The system size is $100 \times 5$ with the Neumann boundary condition. The
			grid sizes are $\Delta x = \Delta y = 0.05$, $\Delta t=0.10$.
		} 
		\label{fig:fhn3}
	\end{figure}  

	Apparently, it makes a sharp contrast with the fusion case. Firstly, the velocity of the spot is quite slow 
	because the parameters are chosen close to the drift bifurcation point, as shown in Fig.\ref{fig:fhn3} (c)(d). Secondly, they almost keep the original shape even at the colliding point, though the distance between them at the nearest point is finite there. Thirdly, annihilation starts to occur in an oscillatory (up and down) manner after they rebound, and the extinction point is located far from the colliding point, as in Fig. \ref{fig:fhn3} (d).
	This looks very counterintuitive because the two spots eventually annihilate even though they look like they are interacting each other very weakly. We call this {\it weakly-interaction-induced-annihilation} (WIIA), which makes a sharp contrast with AF type.
	
	The above example is not an accidental numerical finding but rooted in a deeper structure, as will be clarified in this paper. In one word, it is a consequence of a hidden instability embedded in the localized pattern, which is triggered by the weak collision. Therefore, suppose that any other system shares the same mechanism; then, it shows similar dynamics. 
	In fact, the following three-component system of activator-substrate-inhibitor type presents us another good example of WIIA, as shown in Fig.\ref{fig:gs3}.
	
	\begin{equation}
		\left\{
		\begin{array}{ll}
			u_t  = & \displaystyle{ D_u \triangle u - \frac{uv^2}{1 + f_2 w} + f_0(1-u) }, \vspace*{2mm} \\ 
			v_t  = & \displaystyle{ D_v  \triangle v + \frac{uv^2}{1 + f_2 w} -(f_0 + f_1)v }, \vspace*{2mm} \\ 
			\tau w_t  = & \displaystyle{ D_w  \triangle w + f_3(v - w) },  
		\end{array}
		\right.
		\label{gs3-pde}
	\end{equation}
	\noindent
	where $f_0, f_1, f_2 $, and $f_3$ are all positive parameters related to
	inflow and removal rates of $u$ and $v$ and inhibition and decay 
	rates for the inhibitor $w$. The parameter $\tau$ controls the relaxation time for the inhibitor $w$. The notation $\triangle$ is the 2-dimensional Laplace operator. One of the origins of this model stems from an activator-substrate-inhibitor model developed by Hans Meinhardt \cite{meinhardt_shell} (see Chapter 5), who classified the shell patterns in the framework of reaction diffusion systems. Note that if the inhibitor $w$ is absent, then the system (\ref{gs3-pde}) is reduced to the Gray-Scott model.

	\begin{figure}[htbp]
			\centering
		\includegraphics[width=13cm,clip]{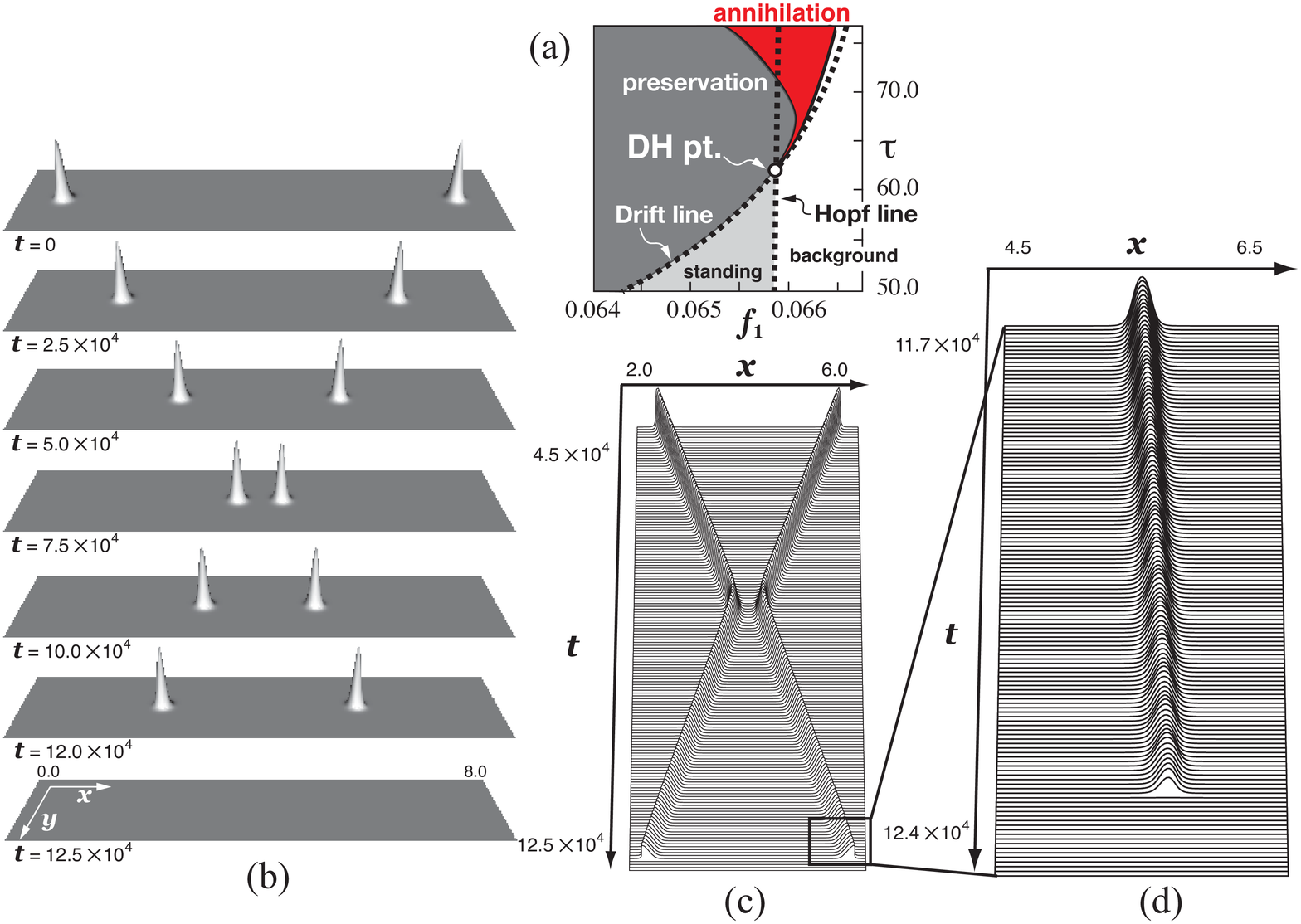}
		\caption{ (a) Phase diagram for the outputs of two colliding traveling spots in \eqref{gs3-pde}
			with respect to $(f_1, \tau)$. The broken (respectively, dotted) line indicates Hopf (respectively, drift) bifurcation for the standing spot. 
			(b) Annihilation behavior near the codimension 2 point of DH type: two slow traveling 
			spots collide and annihilate while keeping some distance. The associated time evolution of the 
			cross section along the $x$-axis is shown in (c). 
			(d) A magnified picture of the oscillatory behavior just before annihilation. 
			Only profiles of the $v$-component are shown. 
			The parameters are set to $(D_u, D_v, D_w) = (2.0 \times 10^{-4}, 1.0 \times 10^{-4}, 
			5.0 \times 10^{-4})$, and $(f_0, f_2, f_3) = (0.05, 0.50, 0.20)$. 
			The system size is $8 \times 2$ with the Neumann boundary condition. The
			grid sizes are $\Delta x = \Delta y = 2^{-6}$, $\Delta t=0.10$.
		} 
		\label{fig:gs3}
	\end{figure}  
	The third component $w$ of the above two systems is crucial to sustain stable traveling patterns in higher-dimensional space and guarantees the coexistence of multiple stable localized traveling patterns in collision dynamics. In fact, the third component $w$ suppresses the elongation and shrinkage orthogonal to traveling direction and keeps them stable as localized patterns. Another advantage of the three-component system is to allow us to control the instabilities and their parametric dependency easily and detect the singularities of high codimension efficiently, as illustrated for pulse generators in \cite{nishiura-teramoto8} (i.e., double-homoclinicity of butterfly type).

	Our goal is to clarify the mathematical mechanism causing the WIIA from a dynamical system point of view, especially in the framework of hidden instabilities triggered by weak interaction. 
	To accomplish the goal, we first have to understand how the traveling pattern can be pushed out of its basin by a weak external perturbation and what type of solutions form a basin boundary of it, which is one of the aforementioned scattors in collision dynamics.
	By closely looking at the behaviors of Figs.\ref{fig:fhn3} and \ref{fig:gs3}, the up and down oscillation dominated the dynamics just before extinction in addition to the fact that the velocity is very slow; therefore, it is expected that a combination of drift and Hopf instabilities might play a crucial role behind the scenes. In fact, this is the case, and a set of unstable traveling breathers constitute the boundary of basin of traveling spots, and the size of its basin can be controlled by how close the parameters are to a codimension 2 point of {\it{drift-Hopf}} (DH) singularity. We will see in Section 3 that the basin size shrinks to zero as appropriate parameters tend to the DH point, which allows a very weak trigger to incur annihilation. 
	
	As already observed, the minimum distance point is not the place where annihilation occurs but is a turning point for the orbit to shift its dynamic regime from stable traveling spot to breathing motion.
	
	The annihilation regime of WIIA type (indicated in Fig. 1 (a) and Fig. 2 (a)) can be obtained as
	one of the unfolding subsets of codimension 2 point of DH type; hence, such a codimension 2 point deserves to be called an organizing center producing a variety of dynamics including annihilation, though it is hidden when the pattern is isolated and only emerges through the interaction with other patterns or external perturbation.
	
	Note that the codimension 2 point is not a generic situation; however, the unfolding subset, in which we observe WIIA robustly, is an open set emanating from the singularity, which will be explained in detail in the subsequent sections. In other words, WIIA is not a rare event, though it looks counterintuitive.

	We will focus on the 1-dimensional case of (\ref{eqn:gd}) and justify the above scenario by reducing a PDE to finite-dimensional ODE dynamics near codimension 2 point of DH type. Such a reduction captures the essence of annihilation dynamics common to both 1-dimensional and 2-dimensional cases. A useful fact about the 1-dimensional case is that we have already many rigorous results of existence and stability properties \cite{doel-2009,heij-2008,nishiura-suzuki-2021}. Moreover, the information on eigenvalues and eigenfunctions around the codimension 2 singularity are available, which is crucial to obtain the reduced ODE system, as described later. It is of course possible to extend the method to higher-dimensional cases, but we leave this for the future work.

	The remainder of the paper is structured as follows.
	In Section 2, after introducing the 1-dimensional model system, we briefly explain how weak interaction induces annihilation by using the reduced ODEs. It turns out that the size of basin boundary of traveling pulse shrinks as parameters approach the DH singularity that causes annihilation of WIIA type. The existence of DH singularity is confirmed numerically in Section 2.3. In Section 3, we derive a finite-dimensional system around the DH singularity and show how the transition from preservation to annihilation occurs as parameters vary. We also compare the resulting dynamics with that of PDEs. Conclusions and outlooks are presented in Section 4.\\
	
	\section{1-dimensinal model system and main results}
	\subsection{Three-component FHN system in 1-dimensional space}
	To gain deeper understanding of the qualitative description in Section 1, especially to understand the role of the unstable breather, we need to reduce the PDE dynamics to the finite-dimensional one and analyze the detailed mechanism behind it. 
	We accomplish this task by using the three-component FHN system in 1-dimensional space (\ref{eqn:gd1}), which completely inherits essential features of annihilation of WIIA type described in the previous section.
	
	The reason for employing the model system (\ref{eqn:gd1}) is firstly that it captures the essence of annihilation dynamics common to both 1-dimensional and 2-dimensional cases. Secondly, we also have many rigorous results of existence and stability properties \cite{doel-2009,heij-2008,ikeda-akama-2017,heij-nishiura-2018,teramoto-heij-2021,nishiura-suzuki-2021} for the 1-dimensional case. Finally, the information on eigenvalues and eigenfunctions around the codimension 2 singularity are computable, which is crucial to obtain the reduced ODE system, as described later. 
	
	The three-component system of the FHN type in 1-dimensional space reads as
	
	\begin{equation}\label{eqn:gd1}
		\left\{
		\begin{array}{ll}
			u_t &=D_u u_{xx} +k_2 u - u^3 -k_3 v - k_4 w + k_1, \\
			\tau v_t &=D_v v_{xx} + u - \gamma v, \\
			\theta w_t &= D_w w_{xx} + u -w.
		\end{array}
		\right.
	\end{equation}
	\noindent
	An annihilation of WIIA type as depicted in Fig. \ref{fig:birdview-codim3} was originally presented in \cite{ueda-teramoto-nishiura-2005} for (\ref{eqn:gd1}), in which the outline of the reduction to ODEs was also presented. In what follows, we shall show a more detailed and complete version of the reduction method, followed by precise analysis of the basin boundary between preservation and annihilation for the reduced system (\ref{eqn:ode-2pulse-p}). This clarifies that the stable manifold of UTB forms a basin boundary of STP and how it behaves as parameters vary around the DH singularity depicted in Fig.\ref{fig:pd-ode} (a).

	\begin{figure}[htbp]
		\centering
		\includegraphics[width=9cm,clip]{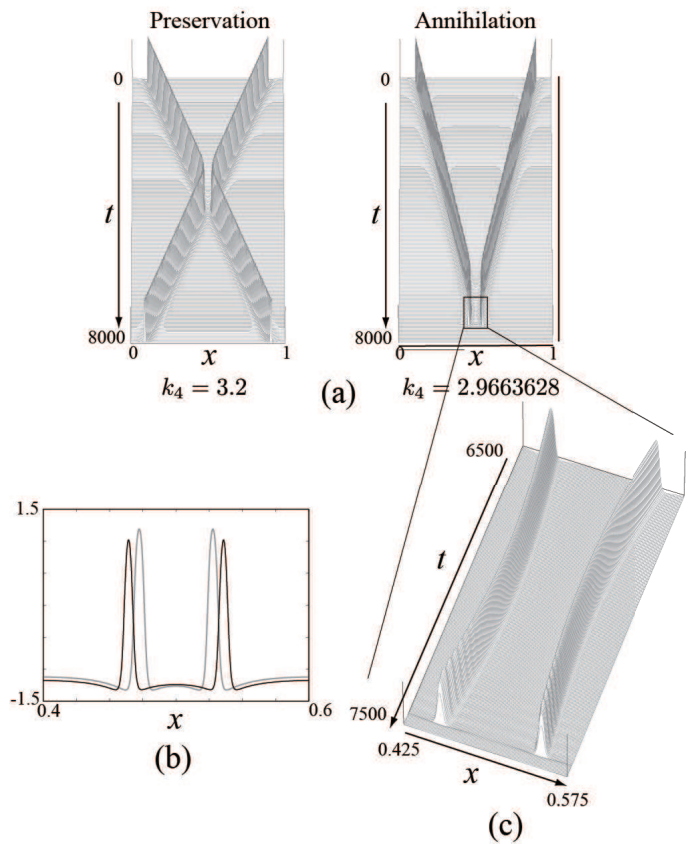}
		\caption{(a) Transition from preservation to annihilation for $\tau=1350$.    Preservation: $k_4=3.2$ (left). Annihilation: $k_4=2.9663628$ (right). (b) Snapshots of the solution profiles corresponding to (a), 
			when the distance between two pulses attains the minimum value for $k_4=2.9663628$ (solid line) and for $k_4=3.2$ (gray line).  
			(c) Magnification near the annihilation point of (a) ($k_4=2.9663628,\tau=1350$). 
			After rebound, up and down motion grows rapidly just before extinction. 
		}
		\label{fig:birdview-codim3}
	\end{figure}
	\begin{figure}[htbp]
		\centering
		\includegraphics[width=12cm,clip]{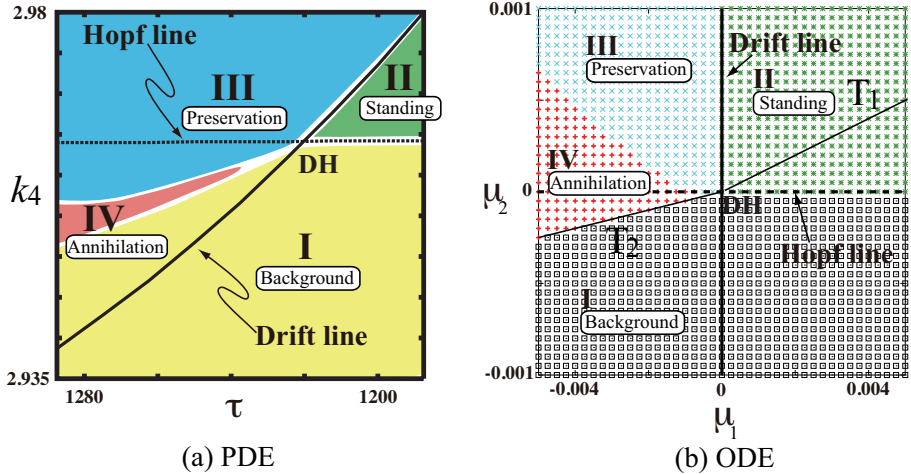}
		\caption{ (a) Phase diagram for the original PDE (\ref{eqn:gd}) with $D_v=0.00065$. 
			Solid and dashed lines indicate drift and Hopf bifurcation
			line, respectively. Stable traveling pulses exist 
			for regions III and IV. I: background state (no patterns). 
			II: standing pulse. III: preservation. IV: annihilation after collision. 
			DH is the codimension 2 point, i.e., drift and Hopf point: 
			$\tau\approx 1220, k_4\approx 2.965$. 
			(b) Phase diagram for the reduced ODE (\ref{eqn:ode-2pulse-p}). The parameters 
			are set to $p_{11}=1.0, p_{12}=1.0, p_{21}=0.05, p_{22}=0.1, M_2=1.0, M_3=1.0$.
			The vertical (horizontal) line $\mu_{1}= 0$ ($\mu_{2} = 0$) is a drift (Hopf) bifurcation from $\text{EP}_{0}$. 
			The two lines $T_{1}$ and $T_{2}$ denote the bifurcation of $\text{EP}_{3}^{\pm}$. 
		}
		\label{fig:pd-ode}
	\end{figure}
	Figure \ref{fig:pd-ode} (a) shows the phase diagram of the outputs for the colliding pulse solutions, in which we draw the two lines of drift and Hopf bifurcations for the single pulse solution. The intersecting point is the DH singularity located at $(k_4,\tau) \approx (2.965, 1220)$. The tongue of the annihilation regime almost touches the DH singularity, though it is not easy to detect it numerically. %
	A typical transition from preservation to annihilation was observed as shown in Fig.\ref{fig:birdview-codim3} (a), where $k_4$ is decreased from $3.2$ to $2.9663628$ with $\tau$ being fixed as $1350$. Note that this set of parameters is outside of the frame of Fig.\ref{fig:pd-ode} (a). This is simply because the annihilation process is difficult to see visually if the parameters are too close to DH singularity.
	Figure \ref{fig:birdview-codim3} (b) shows two profiles of pulses for annihilation (solid black) and preservation (gray) cases when the distance between two pulses takes the minimum value for each case. An interesting point is that the distance for preservation is smaller than for annihilation. 
	Figure \ref{fig:birdview-codim3}(c) shows a magnification near the extinction point. The two traveling pulses collide very weakly and change their directions after collision, then up and down oscillation becomes visible. Finally, the pulses disappear suddenly. 
	Meanwhile, the reduced ODEs show qualitatively the same phase diagram as in Fig.\ref{fig:pd-ode} (b), which is the main theme of the next section.

	At first sight, the manner of extinction in Fig. \ref{fig:birdview-codim3} looks counterintuitive, just like in Figs. \ref{fig:fhn3} and \ref{fig:gs3}. This is partly because it has been believed that only preservation can be observed when the pair of slowly moving pulses interact weakly through the tails; for instance, it was proved by weak perturbation method that this is the case for a special 2$\times$2 reaction diffusion system of excitable type \cite{ei-mimura-nagayama}.

	Historically, the classical 2$\times$2 reaction diffusion system and its variants serve as one of the basic mathematical models for studying pulse dynamics. The existence of traveling pulses via drift bifurcation was established, for instance, by \cite{ikeda-mimura-nishiura-1989,nishiura-mimura-ikeda-fujii-1990}.
	Layer oscillations via Hopf bifurcation and standing breathers emanated from standing pulses have been studied both analytically and numerically \cite{nishiura-mimura-1989,ikeda-nishiura-1994,ikeda-ikeda-2000}. 
	The global bifurcation picture, including the four different pulse behaviors (standing pulses, traveling pulses, standing breathers, and traveling breathers) were studied in the singular limit system \cite{ikeda-ikeda-mimura-2000,mimura-nagayama-ikeda-ikeda-2000,nishi-nishiura-teramoto-2013}. Stable traveling breathers were numerically detected in the narrow parameter region around the DH point \cite{nagayama-ueda-yadome-2010}. %
	
	These previous studies mainly focused on when and how the patterns of breathing type can be observed in various parameter settings, but this is only one side of the coin.
	Namely, they only focused on the half of the unfolding around DH singularity. In contrast to previous studies, we focus on the other half of the unfolding, i.e., super-critical drift and sub-critical Hopf bifurcations near DH point, and show that such a case can be realized for the system (\ref{eqn:gd1}). This unfolding makes a change of the stability properties of bifurcated patterns, i.e., traveling breathers emanated from unstable traveling pulse as a secondary bifurcation are unstable and the basic traveling pulse recovers its stability instead. 
	In the subsequent sections, we rigorously show that such unstable breathers exist and play a role of basin boundary for the annihilation dynamics of WIIA type in the framework of reduced ODEs near DH singularity. 
	
	\subsection{Summary of the main results}  
	It is instructive to present a brief summary of the main results for the reduced system.
	The details are presented in Section 3.
	
	The reduced system for the ``symmetric" head-on collision case describes the evolution of the distance between two pulses $h(t)$, the velocity $v(t)$ (the other one is $-v(t)$), and the amplitude of Hopf instability $A(t) \in \mathbb{R}$. In fact, the growth of $A(t)$ indicates the onset of annihilation.
	The principal part is given by the following system (see Section \ref{sec:ode}): 
	\begin{equation}\label{eqn:ode-2pulse-nojoyo}
		\left \{
		\begin{aligned}
			\dot{v} &= (-\mu_1-p_{11}v^2+p_{12}A^2)v
			-M_2s, \\
			\dot{A} &= (-\mu_2-p_{21}v^2+p_{22}A^2)A
			+M_3s, \\
			\dot{s} &= 2\alpha  (v+M_1s)s,
		\end{aligned}
		\right.
	\end{equation}
	where $\mu_1$ and  $\mu_2$ denote the deviations of parameter values  
	from the codimension 2 point of DH singularity, and $s(t)$ is defined by $s(t) = \exp(-\alpha h(t))$ for $\alpha > 0$. All the coefficients are positive, and $p_{12} p_{21}/ p_{11} p_{22} < 1$ is assumed to guarantee the super- (sub-)criticality of drift (Hopf) bifurcation.
	We also assume that $M_3/M_2> \max\{\sqrt{3p_{21}/{p_{22}}}, 2p_{21}\sqrt{p_{11}p_{12}}/(p_{11}p_{22}-p_{12}p_{21})\}$ (see (S6) in Section 3).
	
	There are six different equilibria for (\ref{eqn:ode-2pulse-nojoyo}) with $s=0$, $\widetilde{\mbox{EP}}_0, \widetilde{\mbox{EP}}_1,  \widetilde{\mbox{EP}}_2^\pm, \widetilde{\mbox{EP}}_3^\pm$, which correspond to standing pulse (SP), 
	standing breather (SB), traveling pulses (TPs), and traveling breathers (TBs) in the PDE sense, respectively. Note that $s=0$ means $h=+\infty$, namely, the above six equilibria represent the list of the ``single" pulse that depends on the parameters.
	Here, $\pm$ indicates the traveling direction. Figure \ref{fig:hb+hb_flow} presents the phase profiles projected to $s=0$ and the parametric dependency on $(\mu_1, \mu_2)$. The equilibrium points in Fig.\ref{fig:hb+hb_flow} without $\ \widetilde{}\ $ indicate that those are critical points restricted to $s=0$. 
	The vertical (horizontal) line $\mu_{1}= 0$ ($\mu_{2} = 0$) is a drift (Hopf) bifurcation from $\text{EP}_{0}$. 
	The two lines $T_{1}$ and $T_{2}$ denote the bifurcation curves of $\text{EP}_{3}^{\pm}$. 
	
	When we consider the symmetric collision, we focus on a special orbit 
	starting from $\widetilde{\mbox{EP}}^{+}_2: (v,A,s)=\left(\sqrt{-\frac{\mu_1}{p_{11}}},
	0, 0\right)$ at $t = -\infty$.
	We are interested in the fate of this orbit after collision. 
	When two pulses collide and rebound without annihilation, it is just a heteroclinic orbit connecting $\widetilde{\mbox{EP}}_2^{+}$ ($t=-\infty$) to $\widetilde{\mbox{EP}}_2^{-}$ ($t=+\infty$). Meanwhile, when annihilation occurs, the orbit starting from $\widetilde{\mbox{EP}}_2^{+}$ collides with the other pulse, and oscillatory instability grows, i.e., the amplitude of $A(t)$ becomes larger and eventually crashes to the background state. Therefore, it is natural to define the preservation and annihilation as follows.
	\begin{definition}
		When $(v,A, s)$ converges to $\widetilde{\mbox{EP}}_2^-$ as $t \rightarrow \infty$, the dynamics is called preservation. 
	\end{definition}
	\begin{definition}
		When $A$ diverges as $t$ is increased, the dynamics is called annihilation. 
	\end{definition}

	\begin{figure}[htbp]
		\centering
		\includegraphics[width=9cm,clip]{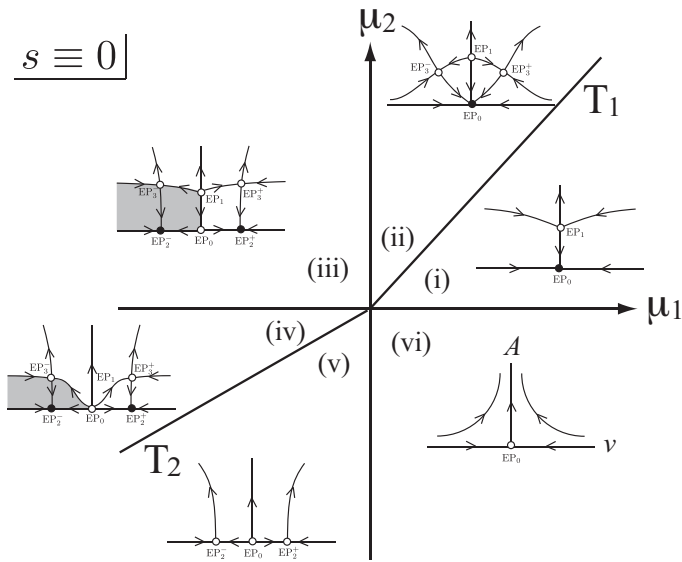}
		\caption{Phase portraits of \eqref{eqn:ode-2pulse-nojoyo} restricted to $s\equiv 0$ under the assumption (S5) in Sec \ref{sec:bifdiag}. 
			Filled (open) circles represent stable (unstable) equilibrium points. The vertical (horizontal) line $\mu_{1}= 0$ ($\mu_{2} = 0$) is a drift (Hopf) bifurcation from $\text{EP}_{0}$. 
			The two lines $T_{1}$ and $T_{2}$ denote the bifurcation curves from which $\text{EP}_{3}^{\pm}$ emanate. }
		\label{fig:hb+hb_flow}
	\end{figure}
	
	Note that the divergence of $A$ may occur at some finite $T$.
	The transition from preservation to annihilation occurs as $(\mu_1, \mu_2)$ varies the region vertically from (iii) to (v) for a fixed negative $\mu_1$ in Fig.\ref{fig:hb+hb_flow}.
	
	Here, we briefly explain how the orbit can be pushed out of the basin of stable traveling pulse and annihilate via weak perturbation. The phase portrait of the single pulse of Fig.\ref{fig:hb+hb_flow} helps us to understand it at intuitive level. The basin of stable traveling pulse $\mbox{EP}_{2}^{-}$ is a shaded region of Fig.\ref{fig:hb+hb_flow}. Suppose the two pulses collide and interact weakly through the tails and still remain in the shaded region; then, loosely speaking, it is a preservation.  As parameter $\mu_2$ is decreased and goes into region (iv), the basin starts to shrink and disappears in region (v). It is quite probable that the orbit can be kicked out of the basin in region (iv) and be extinct, which suggests that an arbitrary small external perturbation becomes a trigger for annihilation because the basin eventually disappears in region (v).

	Our main theorem can be summarized as follows without technical terms;
	see Section 3 for more detailed assumptions and precise statements.
	
	\begin{theorem}[Preservation, annihilation, and shrinkage of basin]
		The dynamics of symmetric head-on collisions is governed by (\ref{eqn:ode-2pulse-nojoyo}). 
		For any fixed sufficiently small $\mu_1 = -\mu_1^0 < 0$, 
		there exists a positive critical value $\mu_2^c > 0$, at which 
		the orbit of (\ref{eqn:ode-2pulse-nojoyo}) starting from $\widetilde{\mbox{EP}_2}^+$ reaches the scattor $\widetilde{\mbox{EP}_3}^-$ after collision. The outcome of an orbit starting from the above (below) $\mu_2^c > 0$ is preservation (annihilation). The stable manifold of $\widetilde{\mbox{EP}_3}^-$ is the separator between preservation and annihilation. The transition from preservation to annihilation occurs due to the shrinkage of the basin of $\widetilde{\mbox{EP}_2}^+$ as $\mu_2$ is decreased.
	\end{theorem}
	
	\subsection{Numerical verification of the existence of codimension 2 bifurcation point}
	In this section, we numerically investigate the existence of DH singularity for the system (\ref{eqn:gd1}) and study the bifurcation structure around it.
	
	\begin{itemize} 
		\item {\it{Parameter setting}}:
		we employ $k_4$ and $\tau$ as bifurcation parameters in the sequel, and
		all other coefficients are fixed as $D_u=5.0\times 10^{-6}$,
		$D_v=6.5\times 10^{-4}$, $D_w=7.5\times 10^{-4}$, $k_1=-3.0, k_2=2.0, k_3=2.0$, 
		$\gamma=8.0$, and $\theta=10.0$, unless otherwise mentioned. For numerical integration, we adopt $\Delta t=1.0\times 10^{-3}$ and $\Delta x =5.0\times 10^{-4}$, and the semi-implicit scheme is used for diffusion terms. 
	\end{itemize}
	
	\begin{figure}[htbp]
		\centering
		\includegraphics[width=11cm,clip]{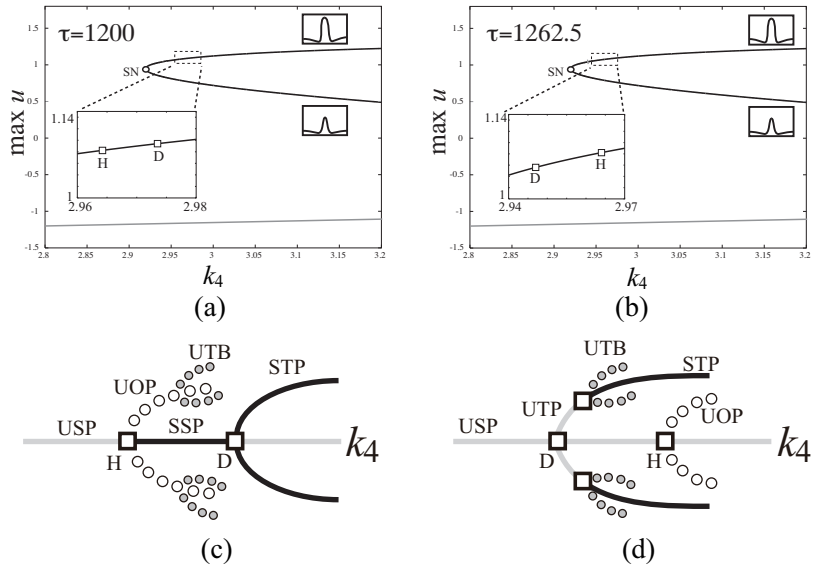}
		\caption{(a)(b) Black and gray lines indicate stationary standing pulse solution
			and uniform state for $\tau=1200$ and $\tau=1262.5$, respectively. 
			D, H, and SN denote drift, Hopf, and saddle-node bifurcation point, respectively. 
			(c)(d) Schematic figure of bifurcation diagram of 
			single pulse near DH point for $\tau=1200$ [(c)] and $\tau=1262.5$ [(d)]. 
			Black and gray lines indicate stable and unstable stationary 
			solutions. Circles indicate unstable periodic solutions.
			SSP: stable standing pulse. USP: unstable standing pulse. 
			STP: stable traveling pulse. UTP: unstable traveling pulse.
			UOP: unstable oscillatory pulse. UTB: unstable traveling breather. 
		}
		\label{fig:codim3-bifdiag-gam8.0}
	\end{figure}

	\begin{figure}[htbp]
		\centering
		\includegraphics[width=11cm,clip]{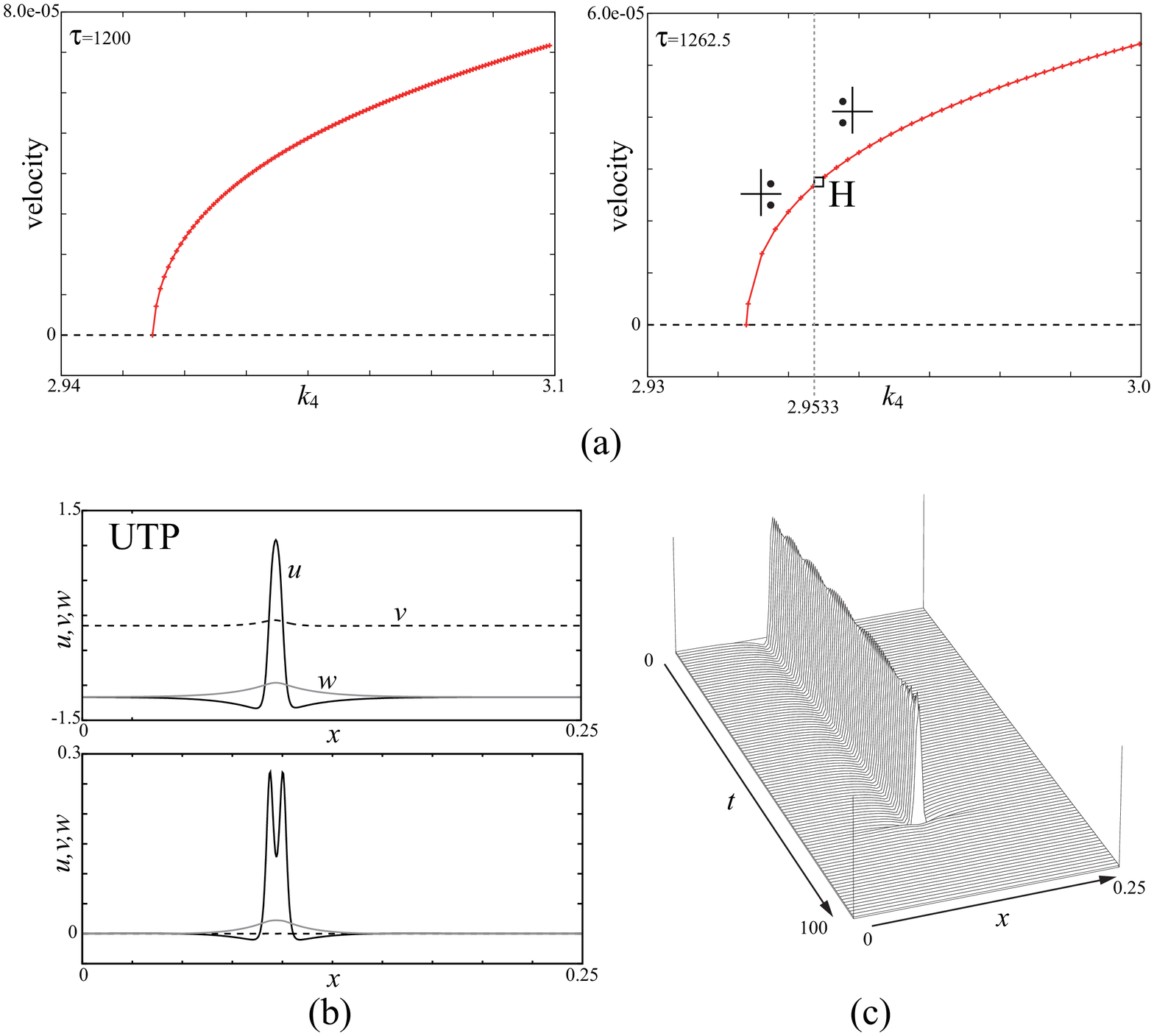}
		\caption{(a) Bifurcation diagram of traveling pulse solution for 
			$\tau=1200$ (left) and $\tau=1262.5$ (right).  The square indicates
			Hopf bifurcation $k_4\approx 2.954$. A pair of eigenvalues 
			crosses the imaginary axis from left to right as $k_4$ is
			decreased. 
			(b) Unstable traveling pulse solution for $k_4=2.9533$ and the 
			real part of the eigenfunction corresponding to the maximum eigenvalue; 
			$\lambda = 1.83\times 10^{-3} \pm 1.937 \times 10^{-1}i$. 
			(c) Responses of unstable traveling pulse to small perturbation of 
			the unstable eigenmode corresponding to the largest eigenvalue. The initial data are
			taken from the unstable branch of traveling pulse solution at $k_4 =
			2.9533$, left side of the Hopf point. Due to Hopf instability,
			oscillatory behavior is observed until extinction occurs. 
			These numerical experiments were done with
			$L=0.25$ with periodic boundary conditions.}
		\label{fig:eigenvec-ptb}
	\end{figure}

	\subsubsection{Global bifurcation diagram of single pulse}
	Figure \ref{fig:codim3-bifdiag-gam8.0} shows two global bifurcation diagrams of single standing pulse of (\ref{eqn:gd1}) for $\tau=1200$ (a) and $\tau=1262.5$ (b) as $k_4$ varies. The secondary bifurcations of drift and Hopf instabilities are depicted in the magnified inlets, and their order is changed as $\tau$ is increased. Therefore, the codimension 2 singularity of DH type occurs in between. A schematic diagram of whole branches including secondary ones for each case is depicted in (c) and (d), from which we can extract key information on the transition from preservation to annihilation. We first take a close look at the solution behaviors around the secondary Hopf bifurcation for $\tau=1262.5$. 
	Figure \ref{fig:eigenvec-ptb} (a) (left) shows a traveling pulse branch via drift instability at $k_4=2.973$ for $\tau=1200, \gamma=8.0,\theta= 10.0$. It is supercritical and inherits the stability of standing pulse. Note that the Hopf instability already occurred in the left side of it (outside of this frame), as shown in Fig.\ref{fig:codim3-bifdiag-gam8.0}(c). Meanwhile, the Hopf instability appears as a secondary bifurcation of the traveling pulse for $\tau=1262.5$ (right), which is responsible for the recovery of stability of traveling pulse. Namely, the bifurcated oscillatory traveling pulse is of subcritical type and unstable, as shown in Fig.\ref{fig:codim3-bifdiag-gam8.0}(d). Prior to the Hopf point, the traveling pulse branch is unstable (UTP), and the real part of the associated unstable eigenfunction has a hollow in the middle, as shown in Fig.\ref{fig:eigenvec-ptb} (b). This suggests that instability of UTP is a type of up and down 
	Oscillation, as shown in (c), and the final destination is the background state.
	
	\subsubsection{Scattor and basin boundary regarding extinction}
	The unstable manifold of UTP is inherited to unstable traveling breather (UTB) due to the subcriticality of Hopf bifurcation; hence, if the orbit would be pushed outside of the basin of stable traveling pulse, then it would follow the behavior of unstable manifold of UTB, i.e., extinction occurs. In fact, the size of the basin of stable traveling pulse shrinks as $k_4$ approaches the Hopf point; therefore, a tiny external perturbation becomes a trigger  of annihilation. We shall see the details of how this annihilation occurs in the reduced ODE system in Section \ref{sec:ode}. It turns out that UTB is a critical saddle solution emerging at the transition from preservation to annihilation, which is called a {\it scattor} in the study of collision dynamics \cite{nishiura-teramoto1,nishiura-teramoto2}. The unstable manifold of UTB plays a role of separator (basin boundary) between preservation and annihilation. 
	
	\section{Reduction to pulse-interaction-ODEs for scattering dynamics}
	\label{sec:ode}
	We describe the dynamics of widely-spaced localized patterns of reaction diffusion systems near DH point by reducing it to
	finite-dimensional ODEs in the general setting. 
	The method to derive ODEs is quite similar to that developed by
	\cite{ei,ei2,ei-mimura-nagayama,ei-mimura-nagayama2,ei-nishiura-ueda}.
	All the assumptions listed below can be checked numerically 
	for (\ref{eqn:gd1}).  
	In this paper, we only perform formal derivation of ODEs and rigorous proof for the construction of an invariant manifold, and the derivation of the flow on it can be done in a parallel manner to those of references mentioned above.
	
	\subsection{Settings and assumptions}
	Let us consider the following general form of an $N$-component
	reaction-diffusion system in 1-dimensional spaces
	\begin{equation}\label{eqn:rd}
		T\bmu_t = D\bmu_{xx} +{\cal F}(\bmu;k),
	\end{equation}
	where $\bmu=(u_1,u_2,\dotsc,u_N)\in {\cal X}:=\{L^2(\mathbb{R})\}^N$, $k\in
	\Bbb{R}$, 
	${\cal F}:{\cal X}\to {\cal X}$ is a smooth function, and
	$D$ and $T:=\{\tau_j\}_{j=1,2,\dots,N}$ are diagonal matrices with positive
	elements. Here, we take $k$ and one of $\tau_i$, e.g., $\tau_n$, as
	bifurcation parameters and put 
	$(k,\tau_n)=(\tilde k+\eta_1, \tilde \tau_n(1+\eta_2))$, 
	where $\eta_1$ and $\eta_2$ are small values. Later, we suppose that 
	$(k,\tau)=(\tilde k, \tilde \tau_n)$ is the DH point (see (S1)). 
	Then, (\ref{eqn:rd}) is rewritten as
	\begin{equation}\label{eqn:0}
		(I+E_{n,n}(\eta_2))\bmu_t={\cal L}(\bmu;\tilde{k}) + \eta_1 g(\bmu),
	\end{equation}
	where 
	${\cal L}(\bmu;\tilde{k})=T^{-1}[D\bmu_{xx} + {\cal F}(\bmu;\tilde{k})]$, 
	${\eta_1 g(\bmu)}={\cal L}(\bmu;\tilde k+\eta_1)-{\cal
		L}(\bmu;\tilde{k})$, $I$ is an
	$n\times n$ identity matrix,
	and $E_{n,n}(\eta)=(e_{ij})$ is an
	$n\times n$ matrix with $e_{nn}=\eta$ and $e_{ij}=0$ otherwise. 
	We suppose ${\cal F}(\bmzero; \tilde{k})=\bmzero$. 
	Let $L={\cal L}'(S(x);\tilde{k})$ be the linearized operator
	of (\ref{eqn:0}) with respect to $S(x)$
	and $\Sigma_c$ be the spectrum of $L$.
	
	The existence analysis of stationary pulse 
	solution in this category was studied in \cite{doel-2009,nishiura-suzuki-2021}. The codimension 2 singularity of DH type was rigorously 
	detected in \cite{heij-2008,nishiura-suzuki-2021}. Therefore, we assume the following properties: 
	\begin{itemize} 
		\item[(S1)]  There exists
		a spatially symmetric 
		stationary standing pulse solution $S(x)$ 
		and it has codimension 2 singularity of DH type at $(k,\tau_n)= (\tilde k,\tilde{\tau}_n)$.
	\end{itemize}
	It has been proven that the tail of the pulse decays exponentially
	\cite{doel-2009,heij-2008,nishiura-suzuki-2021}, and repulsive dynamics occurs through weak interactions (see (S4)). In fact, one of the components with the slowest 
	decay rate dominates the weak interaction dynamics, and hence, we focus on the profile of the slowest
	component. 
	\begin{itemize}
		\item[(S2)]
		$\Sigma_c$ consists of $\Sigma_0=\{0, \pm i\omega_0\} 
		(\omega_0\in \mathbb{R}^+)$ and 
		$\Sigma_1\subset\{z\in \mathbb{C};\Real z<-\gamma_0\}$ for 
		some positive constant $\gamma_0$, and 
		there exist three eigenfunctions $\phi(x)(=S_x(x))$, $\psi(x)$, and $\xi(x)$ such 
		that
		\begin{equation*}
			L\phi = 0, \quad L\psi = -\phi,\quad L\xi = i\omega_0\xi,\qquad
			(\omega_0\in \Bbb{R}),
		\end{equation*}
		where $i$ is the imaginary unit, $\xi$ is an even function, and $\phi$ and $\psi$ are odd functions (see Fig.\ref{fig:eigenvec}). 
		The asymptotic behaviors of the slowest decay component of $S(x)$ and $\text{Re}\ \!\xi(x)$ are given as follows:
		\begin{align*}
			&S(x) \to {\rm e}^{-\alpha\lvert x\rvert}\bma \quad \text{as}\ x\to\pm\infty, \\
			&\text{Re}\ \!\xi(x)\to{\rm e}^{-\alpha \lvert x\rvert}\bmb\quad \text{as}\ x\to\pm\infty,
		\end{align*}
		where $\alpha$ is a positive constant and $\bma, \bmb\in\mathbb{R}^N$.
	\end{itemize} 
	The profiles of $\phi$, $\psi$, and $\xi$ for (\ref{eqn:gd}) at codimension 2 
	bifurcation point DH
	are depicted in Fig.\ref{fig:eigenvec}.
	
	\begin{itemize}
		\item[(S3)]
		Let $L^*$ be the adjoint operator of $L$. There exist 
		$\phi^*$, $\psi^*$, and $\xi^*$ such that $L^*\phi^*=0$, 
		$L^*\psi^*=-\phi^*$, and $L^*\xi^*=-i\omega_0\xi^*$. 
		The asymptotic behaviors of the slowest decay component of $\phi^*(x)$ and $\xi^*$ are given as follows:
		\begin{align*}
			&\phi^*(x) \to \pm{\rm e}^{-\alpha\lvert x\rvert}\bma^* \quad \text{as}\ x\to\pm\infty, \\
			&\text{Re}\ \!\xi(x)^*\to{\rm e}^{-\alpha \lvert x\rvert}\bmb^*\quad \text{as}\ x\to\pm\infty,
		\end{align*}
		where $\bma^*, \bmb^*\in\mathbb{R}^N$.
	\end{itemize}
	
	Under the assumptions (S1)--(S3), we have the following proposition.
	\begin{proposition}\label{prop:eigenvec}
		The functions $\psi$, $\phi^*$, and $\psi^*$ are uniquely determined by
		the normalization 
		\begin{equation*}
			\langle \psi, \phi\rangle_{L^2} = 0,\ \langle \phi, \psi^*\rangle_{L^2} = 1,\
			\langle \psi,\ \psi^*\rangle_{L^2} = 0. 
		\end{equation*}
	\end{proposition}
	Because the proof is quite similar to that 
	given in \cite{ei,ei-mimura-nagayama}, we omit it. 
	
	We take $\xi$ and $\xi^*$ satisfying 
	\begin{equation*}
		\langle \xi, \xi^*\rangle_{L^2} = 1.
	\end{equation*}
	We note that 
	\begin{equation*}
		\langle \psi, \phi^*\rangle_{L^2} = 1,\   
		\langle \phi, \phi^*\rangle_{L^2} = 0
	\end{equation*}
	are satisfied from Proposition \ref{prop:eigenvec}. 
	
	\begin{figure}[htbp]
		\centering
		\includegraphics[width=11cm,clip]{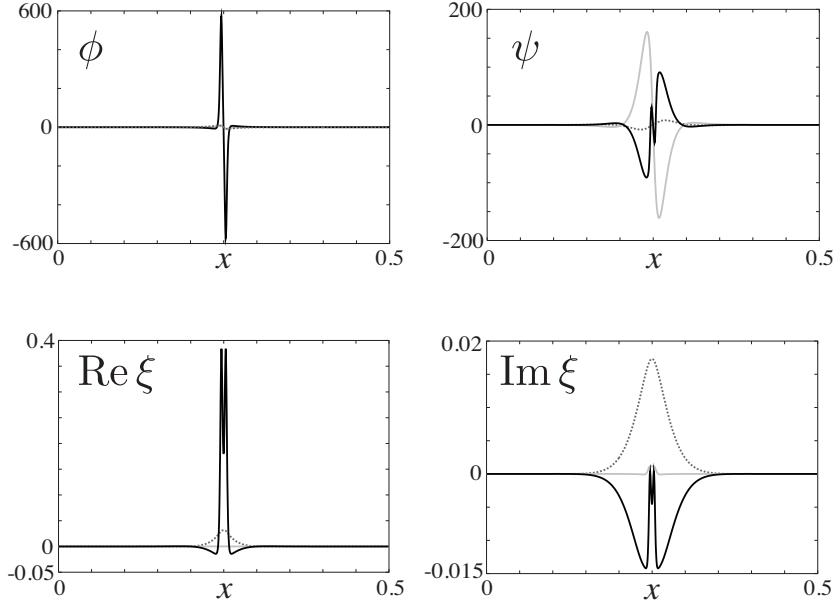}
		\caption{Profiles of $\phi$, $\psi$, and $\xi$. Black solid, gray solid, and dashed 
			lines indicate the $u$-, $v$-, and $w$-component, respectively.}
		\label{fig:eigenvec}
	\end{figure}
	
	\subsection{Derivation of reduced ODE system} 
	The derivation of the ODE system is based on the reduction method developed by Ei and his collaborators 
	\cite{ei,ei-mimura-nagayama,ei-mimura-nagayama2,ei-nishiura-ueda,ei2}, which takes the following two steps: (i) construction of an invariant manifold under the conditions that 
	the parameters are sufficiently close to the DH point and (ii) derivation of an ODE system describing the flows on it. 
	The existence of an invariant manifold for the codimension 2 point of drift-saddle type has been 
	proposed in \cite{ei2}, and the proof for our case, that is, the drift-Hopf type, can be done along the same lines as in the above references. Therefore, we omit the details of it and mainly focus on the derivation of
	the finite-dimensional ODE system on the invariant manifold. 
	
	We define ${\cal L}({\cal X};{\cal X})$ as the set of all bounded 
	operators from ${\cal X}$ to ${\cal X}$. Because ${\cal L}({\cal X}, {\cal L}({\cal X};{\cal X}))$ is 
	identified with ${\cal L}({\cal X}\times {\cal X}; {\cal X})$, we represent $(F'(S)\bmu)'\bmv
	\in {\cal L}({\cal X}; {\cal L}({\cal X};{\cal X}))$ $(\bmu, \bmv\in {\cal X})$ as 
	$F''(S)\bmu\cdot \bmv$, and $F''(S)\bmu\cdot\bmu$ as $F''(S)\bmu^2$ for
	simplicity. The third order derivatives $F'''(S)$ are similarly defined.
	
	We consider the dynamics of two weakly interacting pulses.
	Let $S(x;h)=S(x-h)+S(x)$, 
	$L(h)={\cal L}'(S(x;h))$, $L^*(h)$ be an adjoint operator of 
	$L(h)$, and $\delta(h)=\exp(-\alpha h)$, where $h$ is the 
	distance between two pulses.
	In a similar way to Propositions \ref{prop:1} and \ref{prop:2} of \cite{ei}, we have the 
	following propositions.
	\begin{proposition}\label{prop:1}
		There exist positive constants $C$ and $h^*$ 
		such that for $h>h^*$, the operator $L(h)$ has 
		four eigenvalues $\{\lambda_j(h)\}_{j=1,\dots,4}$ with
		$\lvert \lambda_j(h) \rvert \leq C\delta(h)$ ($j=1,\dots,4$), and 
		$L(h)-i\omega_0 I$  and $L(h)+i\omega_0 I$ have two eigenvalues each,
		$\lambda_5$ and $\lambda_{6}$, and $\lambda_7$ and $\lambda_{8}$, respectively,
		with $\lvert\lambda_j(h) \rvert\leq C\delta(h)$ ($j=5,\dots,8$). 
		Eigenvalues with multiplicities are repeated 
		as many times as their multiplicities indicate. 
		Other spectra of $L(h)$ are on the left-hand side 
		of $\text{Re} z=-\rho_0$ for a positive constant $\rho_0$.
	\end{proposition}
	
	Let $E_1(h)$ and $E_2(h)$ be the generalized eigenspaces corresponding 
	to $\{\lambda_j\}_{j=1,\dots,4}$ and $\{\lambda_j\}_{j=5,\dots,8}$,
	respectively.
	\begin{proposition}\label{prop:2}
		$E_1(h)$ and $E_1^*(h)$ are 
		spanned by functions $\bphi_j(h)(\cdot)$,$\bpsi_j(h)(\cdot)$ 
		and $\bphi_j^*(h)(\cdot)$,$\bpsi^*_j(h)(\cdot)$, respectively, 
		for $j=1,2$, and 
		$E_2(h)$ and $E_2^*(h)$ are 
		spanned by functions $\bxi_j(h)(\cdot)$, $\bar\bxi_j(h)(\cdot)$ 
		and $\bxi^*_j(h)(\cdot)$, $\bar\bxi^*_j(h)(\cdot)$,
		respectively, for $j=1,2$.
		\begin{align*}
			&\bphi_j(h)(x)=\phi(x-p_j)+O(\delta),\\
			&\bpsi_j(h)(x)=\psi(x-p_j)+O(\delta),\\
			&\bphi^*_j(h)(x)=\phi^*(x-p_j)+O(\delta),\\
			&\bpsi^*_j(h)(x)=\psi^*(x-p_j)+O(\delta),\\
			&\bxi_j(h)(x)=\xi(x-p_j)+O(\delta),\\
			&\bxi^*_j(h)(x)=\xi^*(x-p_j)+O(\delta)
		\end{align*}
		hold, where $\delta=\delta(h)$,
		and $O(\delta)$ here means $\|O(\delta)\|_{H^2}\leq C\delta$.
	\end{proposition}
	
	Let $E(h):=E_{1}(h)\oplus E_{2}(h)$ and operators $Q(h)$ and $R(h)$ be the projections from ${\cal X}$ to
	$E(h)$ and $R(h)= Id-Q(h)$, respectively, where $Id$ is the identity 
	on ${\cal X}$. Let $E^\perp(h)=R(h){\cal X}$. Note that $E^\perp(h)$ is 
	characterized by
	\begin{equation*}
		E^\perp(h)=\{\bmv\in {\cal X};\lng\bmv,\bphi_j^*(h)\rng_{L^2}
		=\lng\bmv,\bpsi_j^*(h)\rng_{L^2}=\lng\bmv,\bxi_j^*(h)\rng_{L^2}=0\
		(j=1,2)\}.
	\end{equation*}
	
	Now, we fix $h^*$ large enough such that Propositions\ref{prop:1} and 
	\ref{prop:2} hold.
	Fix $\hat{h}$ with $\hat{h}>h^*$ arbitrarily. Then, we can show that
	there exists a homeomorphic map $\Theta(h)$ from $E^\perp(\hat{h})$ to 
	$E^\perp(h)$ for $h >\hat{h}$ in a similar way to \cite{ei}. 
	Let $p_1$ and $p_2$ be positions of pulses and $h = p_2 - p_1 >0$.
	We take $q_1, q_2\in\mathbb{R}$ and $r_1,r_2\in\mathbb{C}$ with $\lvert \bmq\rvert$
	and $\lvert \bmr\rvert$ sufficiently small.

	Define $E=\text{span}\{\phi,\psi,\xi,\bar{\xi}\}$ and $E^\perp=\{v\in X \lvert
	\langle v, \phi\rangle_{L^2}=\langle v, \psi\rangle_{L^2}=\langle v, \xi\rangle_{L^2}=\langle v, \bar{\xi}\rangle_{L^2}=0\}$.
	Let $\zeta_{ijkl}\in E^\perp$ be 
	the solutions of the equations
	\begin{equation*}
		\begin{aligned}
			L\zeta_{2000}+\Pi_{2000} &= \alpha_{2000}\xi+\bar{\alpha}_{2000}\bar{\xi},\\
			(L-2i\omega_0 I)\zeta_{0200}
			+\Pi_{0200} &= \alpha_{0200}\xi+\bar{\alpha}_{0200}\bar{\xi},\\
			(L-i\omega_0 I)\zeta_{1100}+\Pi_{1100} &= \alpha_{1100}\psi
			+\alpha'_{1100}\phi,\\
			L\zeta_{0110}+\Pi_{0110} &= \alpha_{0110}\xi
			+\bar{\alpha}_{0110}\bar{\xi},\\
			L\zeta_{0001} + \Pi_{0001} &=\alpha_{0001}\xi+\bar{\alpha}_{0001}\bar{\xi},
		\end{aligned}
	\end{equation*}
	where
	\begin{align*}
		&
		\Pi_{2000}=\frac12F''(S)\psi^2+\psi_x,\quad 
		\Pi_{0200}=\frac12F''(S)\xi^2,\quad
		\Pi_{1100}=F''(S)\psi\cdot\xi+\xi_x,\\ 
		&
		\Pi_{0110}=F''(S)\xi\cdot\bar{\xi},\quad
		\Pi_{0001}=g_{1}(S),
	\end{align*}
	and $\alpha_{jklm}$ satisfy the following equations:
	\begin{equation}
		\begin{aligned}
			\nonumber
			&
			\langle \Pi_{2000}-\alpha_{2000}\xi, \xi^* \rangle_{L^2} =0,\quad 
			\langle \Pi_{0200}-\alpha_{0200}\xi, \xi^* \rangle_{L^2} = 0,\quad
			\langle \Pi_{1100}-\alpha_{1100}\psi,\phi^*\rangle_{L^2} = 0,\\
			&
			\langle \Pi_{1100}-\alpha'_{1100}\phi,\psi^*\rangle_{L^2} =
			0,\quad 
			\langle \Pi_{0110}-\alpha_{0110}\xi,\xi^*\rangle_{L^2} = 0,\quad
			\langle \Pi_{0001}-\alpha_{0001}\xi, \xi^*\rangle_{L^2} = 0,\quad\\
			&
			\quad \alpha_{0020}=\bar{\alpha}_{0200}, \quad 
			\alpha_{1010}=\bar{\alpha}_{1010}, \quad \alpha_{1010}=\bar{\alpha}_{1100}. 
			\nonumber
		\end{aligned}
	\end{equation}
	
	We define $S_j(x)$, $\psi_j(x)$, $\xi_j(x)$, and $\zeta_{klmn,j}(x)$ as 
	follows:
	\begin{equation*}
		\begin{aligned}
			&S_j(x) :=S(x-h_j),\quad 
			\psi_j(x) := \psi(x-h_j),\quad
			\xi_j(x) := \xi(x-h_j),\quad \\
			&\zeta_{klmn,j} :=\zeta_{klmn}(x-h_j),
		\end{aligned}
	\end{equation*}
	where $h_1=0$ and $h_2=h$, and we define $\zeta_j(x)\in E^{\perp} (j=1,2)$
	as 
	\begin{equation*}
		\zeta_j(x) := q_j^2\zeta_{2000,j}(x)
		+(r_j^2\zeta_{0200,j}(x)+ q_jr_j\zeta_{1100,j}(x) + \text{c.c.})
		+\lvert r_j\rvert ^2\zeta_{0110,j}(x) + \eta_1\zeta_{0001,j}(x).
	\end{equation*}
	
	Let $p_1$ and $p_2$ be the positions of pulses with $p_1<p_2$, and
	the distance between the two pulses is given by $h=p_2-p_1$.  The quantities $q_1, q_2\in \Bbb{R}$ and  $r_1, r_2\in \Bbb{C}$ are small-amplitude parameters. 
	As explained later, $q_j$ and $\lvert r_j\rvert $ correspond to velocity and 
	oscillatory amplitude of the left and right pulse, respectively. 
	The ansatz is as follows: 
	\begin{equation}
		\bmu(x,t)=\Xi(p_1)\left\{\sum_{j=1,2}\left[
		S_j(x)+q_j\psi_j(x)+(r_j\xi_j(x)+\text{c.c.})+\zeta_j(x) 
		\right]
		+ \Theta(h)\bmw\right\}
		\label{eqn:1}
	\end{equation}
	for $\bmw\in E^{\perp}(\hat{h})$, where 
	$\Xi(y)$ is the translation operator given by
	$(\Xi(y)v)(z)=v(z-y)$ for $v\in L^2$.

	We obtain a finite-dimensional ODE system, describing the dynamics of the
	position and velocity and the amplitude of oscillations near the DH point. 
	\begin{proposition}\label{thm:ode-2pulse}
		Under assumptions (S1)--(S3), as long as $\lvert q_j\rvert$ and  $\lvert r_j\rvert$ 
		are sufficiently small and $h=p_2-p_1$ is sufficiently large, 
		the dynamics of $p_j$, $q_j$, and $r_j$ are governed by the 
		following equations:
		
		\begin{equation}
			\begin{aligned}\label{eqn:ode-final}
				\dot{p}_j&= q_j+(m_{1100}q_jr_j+\text{c.c.})+(m_{1200}q_jr_j^2+\text{c.c.})\\
				&\quad + m_{3000}q_j^3+m_{1110}q_j\lvert r_j\rvert^2+m_{1001}q_j\eta_1+H_j^p(h)+\text{h.o.t.},\\
				\dot{q}_j&= g_{1001}q_j\eta_1 +  g_{1002}q_j\eta_2
				+ (g_{1100} q_j r_j+ \text{c.c.})+g_{3000}q_j^3+ (g_{1200}q_jr_j^2 + \text{c.c.}) \\
				&\quad   + g_{1110}q_j\lvert r_j\rvert^2+H_j^q(h)+\text{h.o.t.},\\
				\dot{r}_j &= i\omega r_j +h_{0001}\eta_1
				+(h_{0101}r_j\eta_1 + c.c ) + h_{2000}q_j^2+(h_{0200} r_j^2 + \text{c.c.})\\
				&\quad+ h_{0110}\lvert r_j\rvert^2
				+ (h_{2100} q_j^2r_j + \text{c.c.})+(h_{0300}r_j^3+ \text{c.c.})\\
				&\quad + (h_{0210}r_j\lvert r_j\rvert^2 + \text{c.c.})+H_j^r(h)+\text{h.o.t.}, \\
			\end{aligned}
		\end{equation}
		where $\dot{\ }$ represents the derivative with respect to $t$,  
		\begin{align*}
			&
			m_{1100}=-\alpha'_{1100},\\ 
			&
			m_{3000}= -\langle {\cal F}''(S)\psi\cdot\zeta_{2000}
			+\frac16 {\cal F}'''(S)\psi^3+\partial_x\zeta_{2000}, \psi^*\rangle,\\
			&
			m_{1200}=-\langle {\cal F}''(S)\psi\cdot\zeta_{0200}
			+ {\cal F}''(S)\xi\cdot\zeta_{1100}
			+\frac12 {\cal F}'''(S)\psi\cdot\xi^2
			+\partial_x\zeta_{0200}, \psi^*\rangle,\\
			&
			m_{1110}=-\langle {\cal F}''(S)\xi\cdot\zeta_{1010}
			+{\cal F}''(S)\bar\xi\cdot\zeta_{1100}
			+{\cal F}''(S)\psi\cdot\zeta_{0110}
			+{\cal F}'''(S)\psi\cdot\xi\cdot\bar\xi
			+\partial_x\zeta_{0110}, \psi^*\rangle,\\
			&
			m_{1001}=-\langle g_1'(S)\psi+\partial_x\zeta_{0001}, \psi^*\rangle,\\
			&
			g_{1100}=\alpha_{1100},\\ 
			&
			g_{3000}= \langle {\cal F}''(S)\psi\cdot\zeta_{2000}
			+\frac16 {\cal F}'''(S)\psi^3+\partial_x\zeta_{2000}, \phi^*\rangle,\\
			&
			g_{1200}=\langle {\cal F}''(S)\psi\cdot\zeta_{0200}
			+ {\cal F}''(S)\xi\cdot\zeta_{1100}
			+\frac12 {\cal F}'''(S)\psi\cdot\xi^2
			+\partial_x\zeta_{0200}
			-\alpha_{1100}'\xi_x, \phi^*\rangle,\\
			&
			g_{1110}=\langle {\cal F}''(S)\xi\cdot\zeta_{1010}
			+{\cal F}''(S)\bar\xi\cdot\zeta_{1100}
			+{\cal F}''(S)\psi\cdot\zeta_{0110}, \phi^*\rangle\\
			&
			g_{1001}=\langle g_1'(S)\psi+\partial_x\zeta_{0001}, \phi^*\rangle,\\
			&
			g_{1002}=\langle E_{n,n}(1)\phi,\phi^*\rangle,\\
			&
			h_{2000}=\alpha_{2000},\ 
			h_{0200}=\alpha_{0200},\ 
			h_{0110}=\alpha_{0110},\
			h_{0001}=\alpha_{0001},\\
			&
			h_{2100}=\langle {\cal F}''(S)\psi\cdot\zeta_{1100}
			+{\cal F}''(S)\xi\cdot\zeta_{2000}
			+\frac12 {\cal F}'''(S)\psi^2\cdot\xi+\partial_x\zeta_{1100}, \xi^*\rangle,\\
			&
			h_{0300}=\langle {\cal F}''(S)\xi\cdot\zeta_{0200}+\frac16 {\cal F}'''(S)\xi^3
			,\xi^*\rangle,\\
			&
			h_{0210}=\langle {\cal F}''(S)\xi\cdot\zeta_{0110}
			+{\cal F}''(S)\bar\xi\cdot\zeta_{0200}+\frac12 {\cal F}'''(S)\xi^2\cdot\bar\xi,
			\xi^*\rangle,\\
			&
			h_{0101}=\langle {\cal F}''(S)\xi\cdot\zeta_{0001}+g_1'(S)\xi,
			\xi^*\rangle,\\
		\end{align*}
		and 
		\begin{align*}
			H_1^p(h)&=-H_2^p(h)
			=M_1\text{e}^{-\alpha h}(1+O(\text{e}^{-\gamma_1 h})),\\
			H_1^q(h)&=-H_2^q(h)
			=-M_2\text{e}^{-\alpha h}(1+O(\text{e}^{-\gamma_1 h})),\\
			\Real H_1^r(h)&= \Real H_2^r(h)
			=M_3\text{e}^{-\alpha h}(1+O(\text{e}^{-\gamma_1 h})),
		\end{align*}
		where $\gamma_1>0$. 
		The constants $M_1$, $M_2$, and $M_3$ are 
		given by 
		\begin{align*}
			&M_1 = -\int_{\Bbb{R}}e^{\alpha x}  \langle \{{\cal F}'(S(x))-{\cal F}'(\bmzero)\}\bma,
			\psi^*(x)\rangle_{L^2} dx, \\
			&M_2 = -2\alpha \langle D\bma, \bma^*\rangle, \ 
			M_3 = 2\alpha\langle D\bma, \bmb^*\rangle. 
		\end{align*}
	\end{proposition}
	
	The proof of Proposition \ref{thm:ode-2pulse} is given in Appendix
	\ref{appen:B1}. Our main result can be rewritten based on (\ref{eqn:ode-final}).
	\\
	\\
	\begin{theorem}\label{cor:1}
		Dynamics of two colliding pulses near DH point can be described as 
		\begin{equation}\label{eqn:ode-2pulse}
			\begin{aligned}
				\dot{v}_1 &= (-\mu_1-p_{11}v_1^2+p_{12}A_1^2)v_1
				-M_2s+\text{h.o.t.},\\
				\dot{A}_1 &= (-\mu_2-p_{21}v_1^2+p_{22}A_1^2)A_1
				+M_3s+\text{h.o.t.},\\
				\dot{v}_2 &= (-\mu_1-p_{11}v_2^2+p_{12}A_2^2)v_2
				+M_2s+\text{h.o.t.},\\
				\dot{A}_2 &= (-\mu_2-p_{21}v_2^2+p_{22}A_2^2)A_2 
				+M_3s+\text{h.o.t.},\\
				\dot{s}&=-\alpha s(v_2-v_1-2M_1s)+\text{h.o.t.},
			\end{aligned}
		\end{equation}
		where $v_j, A_j\in \mathbb{R}^+$,  $v_j=q_j +
		O(\lvert \bmq\rvert \lvert \bmr\rvert)$, 
		$A_j=r_j +
		O(\lvert \bmmu\rvert +\lvert \bmq\rvert ^2+\lvert \bmr\rvert^2)$, $\bmmu=(\mu_1, \mu_2)$ 
		($\lvert \bmmu\rvert=\sqrt{\mu_1^2+\mu_2^2}$), $s=\exp(-\alpha h)$, $p_{11}=-\Real G_{3000},\ p_{12}=\Real G_{1110},\ 
		p_{21}=-\Real H_{2100},\ p_{22}=\Real H_{0210}$.
		$\mu_1 = -\Real G_{1001}\eta_1-\Real G_{1002}\eta_2$,\ 
		$\mu_2 = -\Real H_{0101}\eta_1$, and
		$\mu'_2 =-\Imag H_{0101}\eta_1$. See Appendix \ref{sec:appenA2} for the definitions of $G_{ijkl}$ and $H_{ijkl}$. 
	\end{theorem}
	The proof of Theorem \ref{cor:1} is given in Appendix
	\ref{sec:appenA2}.
	In our example, $\eta_{1} = k_{4}-\tilde{k}_{4} $, $\eta_{2} = (\tau - \tilde{\tau})/\tilde{\tau}$, where 
	$(\tilde{k}_{4}, \tilde{\tau})$ is the DH point shown in Fig.\ref{fig:pd-ode}(a). 
	\\
	\\
	We assume the following: 
	\begin{itemize}
		\item[(S4)]
		The signs of $M_2$ and $M_3$ are positive. 
	\end{itemize}
	This was confirmed numerically for (\ref{eqn:gd1}). 
	It is noted that $M_2>0$ means that two pulses interact repulsively.
	The principal 
	part of (\ref{eqn:ode-2pulse}) is given by
	\begin{equation}\label{eqn:ode-2pulse-full}
		\begin{aligned}
			\dot{v}_1 &= (-\mu_1-p_{11}v_1^2+p_{12}A_1^2)v_1
			-M_2s,\\
			\dot{A}_1 &= (-\mu_2-p_{21}v_1^2+p_{22}A_1^2)A_1
			+M_3s,\\
			\dot{v}_2 &= (-\mu_1-p_{11}v_2^2+p_{12}A_2^2)v_2
			+M_2s,\\
			\dot{A}_2 &= (-\mu_2-p_{21}v_2^2+p_{22}A_2^2)A_2
			+M_3s,\\
			\dot{s} &= -\alpha s(v_2-v_1-2M_1s).
		\end{aligned}
	\end{equation}

	The ODE system \eqref{eqn:ode-2pulse} looks close to the normal form of Hopf-Hopf type discussed in \cite{kuznetsov} (see Chap. 8.6).
	However,  \eqref{eqn:ode-2pulse} contains an $s$-variable so that it is not straightforward to prove that the principal 
	part of \eqref{eqn:ode-2pulse} is conjugate to the full system. 
	In the following, we first focus on the principal part (\ref{eqn:ode-2pulse-full}).  The validity of this 
	approximation will be examined numerically in Sec. \ref{sec:comparison}.

	We are particularly interested in the ''symmetric head-on collision" case, namely, when the two pulses have a mirror symmetry, (\ref{eqn:ode-2pulse-full}) becomes the following 3-dimensional system due to $v_2=-v_1 = v,\ A_2=A_1 = A$, which is the main object we analyze in the section.
	
	\begin{corollary}[Symmetric head-on collision]
		The dynamics of the symmetric head-on collision is described by 
		\begin{equation}\label{eqn:ode-2pulse-p}
			\begin{aligned}
				\dot{v} &= (-\mu_1-p_{11}v^2+p_{12}A^2)v
				-M_2s, \\
				\dot{A} &= (-\mu_2-p_{21}v^2+p_{22}A^2)A
				+M_3s,\\
				\dot{s} &= 2\alpha  (v+M_1s)s. 
			\end{aligned}
		\end{equation}
	\end{corollary}

	\subsection{Bifurcation diagram of a single pulse} \label{sec:bifdiag}
	Prior to the detailed analysis of (\ref{eqn:ode-2pulse-p}), it is instructive to investigate a bifurcation diagram of a single pulse solution. The dynamics of a single pulse is obtained by taking $h=\infty$ or $s=0$ for (\ref{eqn:ode-2pulse-p}). Then, we have
	\begin{equation}\label{eqn:ode-1pulse-p}
		\begin{aligned}
			\dot{v} &= (-\mu_1-p_{11}v^2+p_{12}A^2)v,\\
			\dot{A} &= (-\mu_2-p_{21}v^2+p_{22}A^2)A.
		\end{aligned}
	\end{equation}
	Note that the system (\ref{eqn:ode-1pulse-p}) is the same as the amplitude 
	equations of the truncated
	normal form of Hopf-Hopf bifurcation.
	
	It is known that the phase diagram of equation (\ref{eqn:ode-1pulse-p})
	is qualitatively classified depending on $p_{ij}$ (see \cite{kuznetsov}). 
	We assume that $p_{ij}$
	satisfy the following: 
	\begin{itemize}
		\item[(S5)~] 
		\begin{equation}\label{eqn:joken}
			p_{11},\ p_{12},\ p_{21},\ p_{22}>0,\ p_{12}p_{21}/p_{11}p_{22}<1,
		\end{equation}
	\end{itemize}
	which was confirmed numerically for our system.
	
	The phase portrait of (\ref{eqn:ode-1pulse-p}) with (\ref{eqn:joken}) is shown  in Fig.\ref{fig:hb+hb_flow}
	(see \cite{kuznetsov}). 
	For any $(\mu_1,\mu_2)$, (\ref{eqn:ode-1pulse-p}) has a trivial
	equilibrium point
	\begin{equation*}
		\mbox{EP}_0:(v, A) = (0,0).
	\end{equation*}
	Equilibrium points with $A\not=0$ bifurcate at $\mu_2=0$ (Hopf line of PDE).
	We have $\mbox{EP}_1$ for $\mu_2\geq 0$:
	\begin{equation*}
		\mbox{EP}_1:(v, A)=\left(0, \sqrt{\frac{\mu_2}{p_{22}}}\right).
	\end{equation*}
	Equilibrium points with $v\not=0$ bifurcate at $\mu_1=0$ (drift line of PDE).
	We have $\mbox{EP}^\pm_2$ for $\mu_1\leq 0$.
	\begin{equation*}
		\mbox{EP}^\pm_2:(v, A)=\left(\pm\sqrt{-\frac{\mu_1}{p_{11}}}, 0\right).
	\end{equation*}
	From $\mbox{EP}_1$ and $\mbox{EP}^\pm_2$, the equilibrium points with $v\not=0,
	A\not=0$ 
	bifurcate at 
	\begin{equation*}
		T_1 = \left\{(\mu_1, \mu_2)\left \lvert\ \mu_1=\frac{p_{12}}{p_{22}}\mu_2, \mu_2>0\right. \right\}
	\end{equation*}
	and 
	\begin{equation*}
		T_2 =\left\{(\mu_1, \mu_2)\left \lvert\ \mu_2=\frac{p_{21}}{p_{11}}\mu_1, \mu_1<0\right. \right\},
	\end{equation*}
	respectively. In view of the condition $p_{12}p_{21}/p_{11}p_{22}<1$ 
	in (S5), we have $\mbox{EP}^{\pm}_3$ in regions (ii), (iii), and (iv): 
	\begin{equation*}
		\mbox{EP}^{\pm}_3:(v,A)=\left(\pm\sqrt{\frac{-p_{22}\mu_1+p_{12}\mu_2}
			{p_{11}p_{22}-p_{12}p_{21}}},\ 
		\sqrt{\frac{-p_{21}\mu_1+p_{11}\mu_2}
			{p_{11}p_{22}-p_{12}p_{21}}}\right) .
	\end{equation*}
	The solutions $\mbox{EP}^{\pm}_3$ with $v\not=0$ and $A\not=0$ correspond to traveling breathers that play a key role as a separator. 
	The bifurcation diagrams near the DH point
	shown in Fig.\ref{fig:codim3-bifdiag-gam8.0}(c) (resp.(d))
	are qualitatively equivalent to the behaviors by varying 
	(vi)$\to$ (i)$\to$ (ii)$\to$ (iii)
	(resp. (vi)$\to$ (v)$\to$ (iv)$\to$ (iii)) in Fig.\ref{fig:hb+hb_flow}. 
	Therefore, the ODE system (\ref{eqn:ode-1pulse-p}) with (S5) describes the PDE dynamics of a single pulse solution quite well.

	
	\begin{table}
		\begin{center}
			\begin{tabular}{|c|c|c|c|c|}
				\hline
				Region & $\mbox{EP}_0$ & $\mbox{EP}_1$ &
				$\mbox{EP}_2^\pm$ & $\mbox{EP}_3^\pm$ \\ \hline
				(i) & stable & unstable & --- & --- \\
				(ii) & stable & unstable & --- & unstable \\ 
				(iii) & unstable & unstable & ${\bf stable}$ & unstable \\
				(iv) & unstable & --- & {\bf stable} & unstable \\
				(v) & unstable & --- & unstable & --- \\
				(vi) & unstable & --- & --- & --- \\ \hline
				the associated & & & & \\ 
				dynamics of a single & standing pulse & standing breather & traveling pulse & traveling breather \\ 
				pulse solution & {\bf SP} & {\bf SB} & {\bf TP} & {\bf TB} \\ \hline
			\end{tabular}
			\vspace*{0.3cm}
			\caption{Existence of six different types of equilibria (${\mbox{EP}_0}, {\mbox{EP}_1}, {\mbox{EP}_2}^\pm, {\mbox{EP}_3}^\pm$) 
				for the 2-dimensional system of (\ref{eqn:ode-1pulse-p}) and their stability properties, depending on the 
				parameters $(\mu_1,\mu_2)$ (see Fig.\ref{fig:hb+hb_flow}). 
				The notation "---" represents the non-existence of the solutions. 
				They are associated with the four different pulse behaviors 
				(SP, SB, TP, and TB)
				of a single pulse solution in terms of the original PDEs. 
				The $\pm$ sign for ${\mbox{EP}_2}$ and ${\mbox{EP}_3}$ indicate 
				the propagating direction to the right or left, respectively. 
			}
			\label{Table_EP}
		\end{center}
	\end{table} 
	
	The existence region of these equilibria in the parameter space and their stabilities (in the projected 2-dimensional space $s=0$) are summarized in Table 1.
	We note that $\mbox{EP}_3^{\pm}$ coincides with $\mbox{EP}_2^\pm$ for $(\mu_1, \mu_2)$ on $T_2$,
	i.e., the basin size for $\mbox{EP}_2^\pm$ shrinks as $(\mu_1,\mu_2)$ tends to $T_2$.  
	This is an important property to understand the switching mechanism from
	preservation to annihilation.
	
	\subsection{Definition of preservation and annihilation in the ODE system}
	The equilibria of Table \ref{Table_EP} automatically become those of (\ref{eqn:ode-2pulse-p}) by setting the third component as $s=0$. Namely, the system (\ref{eqn:ode-2pulse-p})
	has a trivial equilibrium point 
	\begin{equation*}
		\widetilde{\mbox{EP}}_0: (v,A,s)=(0,0,0)
	\end{equation*}
	and nontrivial equilibrium points
	\begin{align*}
		&\widetilde{\mbox{EP}}_1:
		(v,A,s)=\left( 0,
		\sqrt{\frac{\mu_2}{p_{22}}}, 0\right),\\
		&\widetilde{\mbox{EP}}^\pm_2: (v,A,s)=\left(\pm\sqrt{-\frac{\mu_1}{p_{11}}},
		0, 0\right),\\
		&\widetilde{\mbox{EP}}^{\pm}_3:(v,A,s)=\left(\pm\sqrt{\frac{-p_{22}\mu_1
				+p_{12}\mu_2}{p_{11}p_{22}-p_{12}p_{21}}},\ 
		\sqrt{\frac{-p_{21}\mu_1+p_{11}\mu_2}
			{p_{11}p_{22}-p_{12}p_{21}}}, 0\right).
	\end{align*}

	\noindent
	Here, we added $\widetilde{\ \ }$ to the equilibria, indicating that they are the solutions of \eqref{eqn:ode-2pulse-p}. Note that their stabilities are different from those in Table \ref{Table_EP} as solutions of \eqref{eqn:ode-2pulse-p}.
	When we consider the symmetric head-on collision, the initial value is specified as 
	$\widetilde{\mbox{EP}}^{+}_2: (v,A,s)=\left(\sqrt{-\frac{\mu_1}{p_{11}}},
	0, 0\right)$, meaning that the orbit starts from $\widetilde{\mbox{EP}}^{+}_2$ at $t = -\infty$. We are interested in the fate of this orbit after collision. 
	
	We define $v_{\text{ep2}}$, $v_{\text{ep3}}$, and $A_{\text{ep3}}$ by
	\begin{equation*}
		v_{\text{ep2}}:= \sqrt{\frac{-\mu_{1}}{p_{11}}},\quad 
		v_{\text{ep3}}:= \sqrt{\frac{-p_{22}\mu_1+p_{12}\mu_2 }{p_{11}p_{22}-p_{12}p_{21}}},\quad 
		A_{\text{ep3}} := \sqrt{\frac{-p_{21}\mu_1+p_{11}\mu_2}
			{p_{11}p_{22}-p_{12}p_{21}}}.
	\end{equation*}
	The equilibrium point $\widetilde{\mbox{EP}}_2^\pm$ is given by $(v, A, s) = (\pm v_{\text{ep2}}, 0, 0)$, and 
	$\widetilde{\mbox{EP}}_3^\pm$ is given by $(v, A, s) = (\pm v_{\text{ep3}}, A_{\text{ep3}}, 0)$. 
	We define the following functions:
	\begin{equation}\label{eqn:GH}
		\begin{aligned}
			G(v)&:= \sqrt{(\mu_{1}+p_{11}v^2)/p_{12}}\quad (\mu_{1}+p_{11}v^{2}\geq 0), \\
			H(v)&:= \sqrt{(\mu_{2}+p_{21}v^2)/p_{22}}\quad (\mu_{2}+p_{21}v^{2}\geq 0),
		\end{aligned}
	\end{equation}
	where $A=G(v)$ and $A=H(v)$ satisfy $\dot{v}=0$ and $\dot{A}=0$ when $s=0$, respectively.

	Let us define the terminologies ``preservation'' and ``annihilation''  
	for the ODE system. The annihilation contains a large deformation, so we can prove something about its onset here, namely, the amplitude of oscillation $A$ grows and the orbit does not come back to any equilibrium point of (\ref{eqn:ode-2pulse-p}).
	Therefore, it is reasonable to use the amplitude of 
	oscillation $A$ as a criterion for annihilation 
	because annihilation is caused by subcritical Hopf bifurcation. 
	
	\begin{definition}[Preservation and annihilation]
		When the solution converges to $\widetilde{\mbox{EP}}_2^-$ as $t\to +\infty$ for a given set of parameters, we call it ``preservation''. On the contrary, when the variable $A$ of the solution diverges, we call it ``annihilation''. 
	\end{definition}
	\noindent
	
	\begin{figure}[htbp]
		\centering
		\includegraphics[width=6cm,clip]{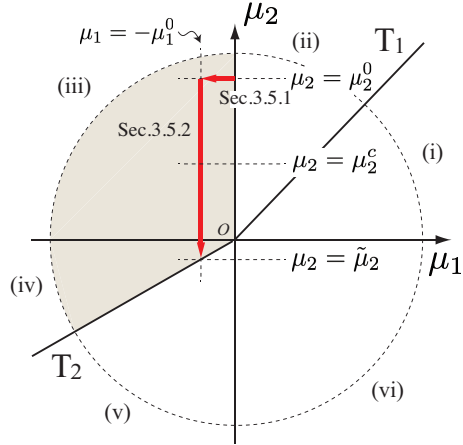}
		\caption{In Sec.\ref{sec:m2fix}, $\mu_2$ is fixed at $\mu_2^0$, and we change $\mu_1$ as a 
			control parameter. We will show that preservation occurs at some $\mu_1 = -\mu_1^0$ for $\mu_2=\mu_2^0$. 
			In Sec.\ref{sec:m1fix}, we take  $\mu_2$ as a control parameter and show that transition 
			from preservation to annihilation occurs at some parameter $\mu_2=\mu_2^c$. }
		\label{fig:parameter-move2}
	\end{figure}
	\subsection{Symmetric head-on collision}
	Now, we are ready to study (\ref{eqn:ode-2pulse-p}), especially the parametric dependency of the orbit starting from $\widetilde{\mbox{EP}}^{+}_2$ at $t = -\infty$, which is associated with a stable traveling pulse in regions (iii) and (iv) in the parameter space $(\mu_1,\mu_2)$ of Fig.\ref{fig:parameter-move2} (see also Fig. 5). We shall show that preservation is observed in region (iii), and the transition from preservation to annihilation occurs at some point in region (iv) as $\mu_2$ is decreased. 
	The equilibrium point  $\widetilde{\text{EP}}_{2}^{+}$  has a positive eigenvalue $2\alpha v_{\text{ep2}}$ for $\mu_{2} \geq \tilde{\mu}_{2}$.  
	In the following, as a convention,  
	we take an initial condition on the corresponding unstable manifold of $\widetilde{\text{EP}}_{2}^{+}$, unless stated explicitly. 
	In Sec.\ref{sec:m2fix}, we show the transition from standing to preservation as  $\mu_1$ is decreased for a fixed $\mu_2=\mu_2^0$. 
	In Sec.\ref{sec:m1fix}, we fix $\mu_1=-\mu_1^0$, and $\mu_2$ is decreased from $\mu_2^0$ to the line $T_2$. 
	We show the existence of the critical point $\mu_2^c$ at which 
	the solution orbit lies on the stable manifold of $\widetilde{\mbox{EP}}_3^-$ separating two regions of preservation and annihilation. 
	
	We define the following regions: 
	\begin{equation*}
		\begin{aligned}
			&D_0 := \{(v, A, s)\ \lvert\ v+M_1s=0,  s\geq 0\},\\
			&D_+ := \{(v, A, s)\ \lvert\ v+M_1s>0,  s\geq 0\},\\
			&D_- := \{(v, A, s)\ \lvert\ v+M_1s<0, s\geq 0\}. 
		\end{aligned}
	\end{equation*}
	We see from \eqref{eqn:ode-2pulse-p} that $s$ increases and decreases when the solution is in $D_+$ and $D_-$, 
	respectively. The two traveling pulses start to repel each other when the solution orbit moves from $D_+$ and $D_-$. 
	
	Before going into the detail, we briefly outline the proof here. First, let us define the time $T$ as the first hitting time of the solution at $D_{0}$, where we take $T=\infty$ 
	when the solution does not hit it. 
	We also define the time $T'$ as the first hitting time when the variable $v$ reaches $-v_{\text{ep3}}$, 
	where we take $T'=\infty$ when the solution does not hit it. 
	The preservation dynamics is roughly divided by the following three steps. 
	In the first step ($t\in [-\infty, T]$), the solution trajectory takes off from $\widetilde{\text{EP}}_{2}^{+}$, and $s$ increases and reaches the maximum value when the solution touches $D_{0}$ (see the 
	proof of Lemma \ref{lem:maxsadded}). 
	In the second step ($t\in (T, T']$), $s$ decreases and approaches the $s=0$ plane. 
	In the final step ($t\in (T', \infty)$), the solution converges to $\widetilde{\text{EP}}_{2}^{-}$. 
	To prove the existence of the preservation parameter regime and the heteroclinic orbit 
	from $\widetilde{\text{EP}}_{2}^{+}$ to $\widetilde{\text{EP}}_{3}^{-}$ in the validity regime of weak interaction,  
	we need to estimate $\lvert v\rvert$, $A$, and $s$ for each step. 
	
	The following lemma will be used to estimate the maximum of $s$, i.e., the minimum distance of the colliding 
	pulses. 
	\begin{lemma}\label{lem:sstar}
		Let $v^{*}$ be a positive constant. The function 
		$$
		f(s)=\frac{16\alpha(v^*)^2}{M_2}-s\log\left(\frac{2v^*+\lvert M_1\rvert s}{v^*+\lvert M_1\rvert s}\right)
		$$
		is monotonically decreasing. There exists a unique $s^*$ such that $f(s^*)=0$ 
		with $s^*=16\alpha(v^*)^2/(M_2\log 2)+O((v^*)^3)$ if $s$ satisfies $\lvert s/M_{1}v^{*}\rvert < 1$.
	\end{lemma}

	\begin{proof}
		Because $2v^*+\lvert M_1\rvert s> v^* + \lvert M_1\rvert s$, it holds that
		$$
		\frac{d}{ds}f(s)=-\log(2 v^*+\lvert M_1\rvert s)-\frac{\lvert M_1\rvert s}{2v^*+\lvert M_1\rvert s}+\log(v^*+\vert M_1\rvert s)
		+\frac{\lvert M_1\rvert s}{v^*+\lvert M_1\rvert s}<0, 
		$$ 
		where we used the fact that the function $\log(x+c) + c/(x+c)$ ($c>0$) is monotonically increasing for $x>0$. 
		From the condition $\lvert s/M_{1}v^{*}\rvert<1$, the Taylor expansion of the right-hand side of $f$ reads as
		$$
		f(s)=\frac{16\alpha(v^*)^2}{M_2}-s\left\{\log 2 -\frac{\lvert M_1\rvert s}{2v^*}+O\left(\left(\frac{s}{v^*}\right)^2\right)\right\}.
		$$
		Thus, $s^*=16\alpha(v^*)^2/(M_2\log 2)+O((v^*)^3)$.
	\end{proof}

	The following lemma gives an upper bound of $s$ when the solution is in $D_0\cup D_+$ and 
	the solution trajectory of \eqref{eqn:ode-2pulse-p} enters $D_-$ in a finite time. 
	\begin{lemma}\label{lem:maxsadded}
		Let assumptions (S1)--(S5) hold.  
		Assume that $A(t)<\widetilde{A}$ for some $\widetilde{A}=O(\lvert \bm{\mu}\rvert ^{1/2})$ and $0<t\leq T'$. 
		Then, it holds that
		\begin{enumerate}
			\item $v=O(\lvert \bm{\mu}\rvert^{1/2})$ and $s(t) = O(\lvert\bm{\mu}\rvert)$ for $t\leq T'$,
			\item  $T$ is finite, and
			\item $s(t)$ attains the maximum value at $t=T$. 
		\end{enumerate}
	\end{lemma}
	\begin{proof}
		When $t<T'$, from the conditions $A(t)<\widetilde{A}$ and $M_{2}>0$, it holds that $\lvert v(t)\rvert <\tilde{v}:= \sqrt{\mu_{1} +p_{12}\widetilde{A}^{2}}/\sqrt{p_{11}}
		=O(\lvert\bm{\mu}\rvert^{1/2})$
		for $t< T'$ (Fig.\ref{fig:yz}(a)).  That is, $\lvert v(t)\rvert = O(\lvert\bm{\mu}\rvert^{1/2})$ for $t<T'$. 
		
		We take $v^{*}=\tilde{v}$ and $s^{*}$ as given in Lemma \ref{lem:sstar}. 
		Assume that there exists $t_1 \leq T$ such that $s(t_1)=s^*/2$. 
		From the definition of $T$,  $v(t)+M_1s(t)\geq 0$ holds for 
		$t\leq t_1 (\leq T)$. 
		We will prove that $v(t)+M_1s(t)=0$ before $t$ reaches 
		$t_2:=t_1+\{\log(2v^*+\lvert M_1\rvert s^*)-\log(v^*+\lvert M_1\rvert s^*)  \}/(2\alpha v^*)$. 
		From 
		(\ref{eqn:ode-2pulse-p}), we have 
		\begin{equation} \label{eqn:c-1added}
			\frac{ds}{dt} = 2\alpha (v +M_1s)s \leq 2\alpha v^* s+2\alpha \lvert M_1\rvert s^2, 
			\quad t_1\leq t. 
		\end{equation}
		Integrating both sides of (\ref{eqn:c-1added}) during $[t_{1}, t]$, we have
		\begin{equation}\label{eqn:estimofs}
			s(t) \leq \frac{v^*s^*}{(2v^*+\lvert M_1\rvert s^*)\exp\left\{-2\alpha v^*(t-t_1)\right\}- \lvert M_1\rvert s^*}.
		\end{equation}
		Therefore, $s(t)\leq  s^*$ for $t\leq t_2:=t_1+\{\log(2v^*+\lvert M_1\rvert s^*)-\log(v^*+\lvert M_1\rvert s^*)  \}/(2\alpha v^*)$, 
		where we used properties that 
		the right-hand side of \eqref{eqn:estimofs} is monotonically increasing and 
		attains $s^*$ at $t=t_2$. 
		Meanwhile, for $t_1\leq t$, it follows that 
		\begin{equation}\label{eqn:c-2added}
			\begin{aligned}
				\dot{v} +M_1\dot{s}&= (-\mu_1 - p_{11}v^2 +p_{12}A^{2})v - M_2 s +2\alpha M_1(v+M_1s)s\\
				&\leq  (2\alpha M_1v +\alpha M_1^2 s) s  - M_2s +  O(\lvert\bm{\mu}\rvert^{3/2})  \\
				& \leq - M_2 s /2  \\
				&\leq -M_2 s^*/4,
			\end{aligned}
		\end{equation}
		where we used  $A <\widetilde{A} = O(\lvert\bm{\mu}\rvert^{1/2})$, $\lvert v\rvert = O(\lvert \bm{\mu}\rvert^{1/2})$, 
		and  $s^* = O((v^*)^2) = O(\lvert \bm{\mu}\rvert)$. 
		Integrating both sides of (\ref{eqn:c-2added}) 
		during $t\in[t_1,t_2]$,
		\begin{align*}
			v (t_2) +M_1s(t_2)&\leq  v(t_1)+\lvert M_1\rvert s(t_1) -M_2 s^* (t_2-t_1)/4\\
			&= v(t_1) +\lvert M_1\rvert s(t_1)-\frac{M_2 s^*}{8\alpha v^*} \log\left(\frac{2 v^*+\lvert M_1\rvert s^*}{v^*+\lvert M_1\rvert s^*}\right) \\
			&\leq  2v^*-\frac{M_2 s^*}{8\alpha v^*} \log\left(\frac{2 v^*+\lvert M_1\rvert s^*}{v^*+\lvert M_1\rvert s^*}\right) \\
			&= 0.
		\end{align*}
		Here, we have used Lemma \ref{lem:sstar} and $\lvert M_1\rvert s(t_1)=\lvert M_{1}\rvert s^{*}/2\leq v^*$.
		This implies that $T\leq t_{2}$, i.e., $T$ is finite. 
		The variable $s$ has the maximum value for $t = T$ because the sign of $v+M_{1}s$ determines the sign 
		of $\dot{s}$. Thus, $s \leq s^{*}$ holds, that is, $s = O((v^{*})^{2}) = O(\lvert \bm{\mu}\rvert)$ for $t\leq T$. 
		Because $s$ decreases in $D_{-}$, $s(t)\leq s(T)$ for $t \geq T$, that is,  $s = O(\lvert \bm{\mu}\rvert)$ for $t\leq T'$. 
	\end{proof}

	\subsubsection{Existence of preservation region --Decreasing $\mu_1$ from the drift line ($\mu_{1}=0$) with $\mu_2=\mu_2^0>0$ being fixed--}\label{sec:m2fix}
	Suppose that $\mu_2=\mu_2^0>0$ is fixed and decreasing $\mu_1$ from the drift line ($\mu_{1}=0$) 
	(Fig.\ref{fig:parameter-move2}). 
	The dynamics of (\ref{eqn:ode-2pulse-p}) around $\mu_1=0$ is 
	described as follows:
	\begin{equation}\label{eqn:reducemu}
		\begin{aligned}
			&\dot{s}=2\alpha s (v+M_1s) + \text{h.o.t.},\\
			&\dot{v}=(-\mu_1-p_{11}v^2)v-M_2s + \text{h.o.t.}
		\end{aligned}
	\end{equation}
	These equations are obtained by 
	applying the center manifold theory around the drift point
	$\mu_1=0$. The variable $A$ is described by the following 
	algebraic equation on the center manifold: 
	\begin{equation}\label{eqn:cmf}
		A = \frac{-2M_1M_2\alpha}{\mu_2^3} s^2 + \frac{2M_2\alpha}{\mu_2^2} sv
		- \frac{M_2}{\mu_2} s.
	\end{equation}
	The derivation of \eqref{eqn:cmf} is given in Appendix \ref{sec:appendc}. 
	
	\begin{lemma}\label{lem:Acmf}
		Let assumptions (S1)--(S5) hold and $\mu_2=\mu_2^0>0$ be fixed. If $\mu_{1}<0$ is taken sufficiently close to $0$, 
		then $s(t) = O(\lvert \mu_{1}\rvert)$ for all $t>0$. 
	\end{lemma}
	\begin{proof}
		Because $M_{2}>0$, $v(t) \leq v(0) = O(\lvert \mu_{1}\rvert^{1/2})$ for $t\in [0, T]$. 
		Let $v^{*} = v(0)$ and $s^{*}$ be as given in Lemma \ref{lem:sstar}. 
		From \eqref{eqn:reducemu}, 
		\begin{equation}\label{eqn:c-2added}
			\begin{aligned}
				\dot{v} +M_1\dot{s}&= (-\mu_1 - p_{11}v^2 )v - M_2 s +2\alpha M_1(v+M_1s)s + \text{h.o.t.}\\
				&\leq  (2\alpha M_1v +\alpha M_1^2 s) s  - M_2s + \text{h.o.t.}\\
				& \leq - M_2 s /2  \\
				&\leq -M_2 s^*/4,
			\end{aligned}
		\end{equation}
		where we used  $v = O(\lvert \mu_{1}\rvert^{1/2})$ and  $s^* = O((v^*)^2) = O(\lvert\mu_{1}\rvert)$ for $t<T$. 
		By the same arguments in the proof of Lamma \ref{lem:maxsadded}, we find that $s(T) = O(\lvert \mu_{1}\rvert)$. 
		Because $s$ attains the maximum value at $t=T$, $s(t)= O(\lvert \mu_{1}\rvert)$ for all $t>0$. 
	\end{proof}
	\begin{theorem}[preservation]\label{thm:preservation}
		Let assumptions (S1)--(S5) hold and $\mu_2=\mu_2^0>0$ be fixed. If $\mu_{1}<0$ is taken sufficiently close to $0$, the solution converges to $\widetilde{\mbox{EP}}_2^- = (-\sqrt{-\mu_1/p_{11}},0,0)$.
	\end{theorem}
	\begin{proof}
		From the center manifold theory in Appendix B, 
		$\lvert v(t)\rvert$ and $s(t)$ can be taken as arbitrarily small for $t\leq T'$  by taking $\lvert \mu_{1}\rvert$ sufficiently smaller 
		than $\mu_{2}$. 
		Thus, from Lemma \ref{lem:Acmf} and \eqref{eqn:cmf}, $A(t)<A_{\text{ep3}}$ holds for $t>0$  
		because $A=O(\lvert \mu_{1}\rvert)$ and $A_{\text{ep3}} =\sqrt{(-p_{21}\mu_{1}+p_{11}\mu_{2})/(p_{11}p_{22} - p_{12}p_{21})}\gg \lvert \mu_{1}\rvert$ 
		for sufficiently small $\lvert \mu_{1}\rvert $. 
		Thus, from Lemma \ref{lem:maxsadded}, we find that the solution crosses $D_{0}$ in finite time. 
		
		Because $\lvert \mu_{1}\rvert, \lvert v\rvert, A \ll 1$, 
		\begin{align*}
			\dot{v} + M_1\dot{s}&= (-\mu_1-p_{11}v^2 + p_{12}A^{2})v-M_2s \\
			&=(-\mu_1-p_{11}v^2+p_{12}A^{2})(-M_1s)-M_2s \\
			&\leq -\frac{M_{2}s}{2}<0. 
		\end{align*}
		This means that the solution belongs to $D_{-}$ for $t>T$, and there exists $\varepsilon>0$ and $\delta>0$ such 
		that $v+M_{1}s<-\varepsilon$ for  $t>T+\delta$. 
		That is,  $s\to 0$ as $t\to \infty$ because $\dot{s}<0$ for $t>T$, and 
		$\widetilde{\mbox{EP}}_2^-$  for $\mu_{1}<0$ 
		(see the phase portrait for $s\equiv 0$ in Fig.\ref{fig:hb+hb_flow} (iii)).  
	\end{proof}
	This theorem indicates that the drift line $\mu_1=0$ is the boundary between the standing and preservation regions. 
	We remark that because $s^*$ and $v^*$ can be arbitrarily small as $\mu_1\to 0$, the maximum 
	value of $A$ tends to $0$ as $\mu_1\to 0$. This implies that 
	no annihilation occurs for sufficiently small $\lvert\mu_1\rvert$ when $\mu_2 >0$ is fixed. 
	
	\subsubsection{Transition from preservation to annihilation --Decreasing $\mu_2$ from $\mu_2^0$ with $\mu_1=-\mu_1^0<0$ being fixed--}\label{sec:m1fix}
	
	Next, we will show that there exists a critical value $\mu_2^{c}$ at which the solution trajectory lies on 
	the stable manifold of $\widetilde{\text{EP}}_{3}^{-}$ (denoted by $\Pi$ in Fig.\ref{fig:yz}). 
	As shown in the phase portrait in Fig.\ref{fig:hb+hb_flow}, it is plausible that $A$ diverges when $\mu_2$ is 
	taken close to $\tilde{\mu}_{2}$ on $T_2$ in Fig.\ref{fig:parameter-move2}. 
	Hence, it follows almost immediately from Theorem \ref{thm:preservation} and the continuity 
	of the solution with respect to the parameters that there exists a critical $\mu_2^c$ where the solution trajectory crosses the stable manifold 
	of $\widetilde{\text{EP}}_{3}^{-}$ between $\mu_{2}^{0}$ and $\tilde{\mu}_{2}$. 
	However, because the reduced ODE is valid in the weak interaction regime, we need to show that 
	$v = O(\lvert\bm{\mu}\rvert^{1/2}), A =O(\lvert\bm{\mu}\rvert^{1/2})$, and $s=O(\lvert\bm{\mu}\rvert)$ for $t>0$ when 
	$\mu_2\in [\mu_2^0, \mu_2^c]$. 
	For this purpose, we introduce an invariant set $X\subset \mathbb{R}^3$ satisfying the property that the orbit diverges once it enters $X$.
	
	Before defining $X$, we examine the number of common points 
	of the curve $A=H(v)$ (recall the definition (\ref{eqn:GH})) and the line $A = -a_{\infty}(v +v_{\text{ep3}}) + A_{\text{ep3}}$ in 
	the region $v\leq -v_{\text{ep3}}$ and $A>0$, 
	where $a_\infty$ is a constant in the interval $(\sqrt{p_{21}/p_{22}}, \sqrt{p_{11}/p_{12}}]$. 
	Because $H(-v_{\text{ep3}}) = A_{\text{ep3}}$, the equation 
	$H(v) = -a_{\infty}(v +v_{\text{ep3}}) +A_{\text{ep3}}$ should have a solution $v = -v_{\text{ep3}}$. 
	We denote the other solution by $-v_{X}$ (see the magnified inset in Fig.\ref{fig:yz} (a)), where 
	\begin{align*}
		v_{X} :=  \frac{p_{22}(a_{\infty}^{2} v_{\text{ep3}}^{2} - 2a_{\infty} v_{\text{ep3}}A_{\text{ep3}} + A_{\text{ep3}}^{2}) - \mu_{2}}
		{(a_{\infty}^{2}p_{22} - p_{21})v_{\text{ep3}}} = O(\lvert \bm{\mu}\rvert^{1/2}). 
	\end{align*}
	We also define
	\begin{equation}\label{eqn:Ax}
		A_{X} := \sqrt{\frac{\mu_{2} +p_{21} v_{X}^{2}}{p_{22}}}= O(\lvert\bm{\mu}\rvert^{1/2}). 
	\end{equation}
	The equation $H(v) = -a_{\infty}(v +v_{\text{ep3}}) + A_{\text{ep3}}$ has a 
	multiple root $v=-v_{\text{ep3}}$ at $\mu_{2} = \mu_{2}^{*}$, where 
	$$
	\mu_{2}^{*} = \left(-p_{21}  + \frac{p_{21}^{2}}{p_{22}a_{\infty}^{2}}\right) v_{\text{ep3}}^{2}. 
	$$
	By direct computation, we have
	\begin{align*}
		\frac{d}{d\mu_{2}}\left(\left.\frac{d H}{dv}\right\lvert_{v=-v_{\text{ep3}}}\right) &= -\frac{(p_{11}p_{22}-p_{12}p_{21}) \mu_{1}}
		{2(-p_{21}\mu_{1}+p_{11}\mu_{2})^{3/2}(-p_{22}\mu_{1}+p_{12}\mu_{2})^{1/2}} >0. 
	\end{align*}
	Thus, $(dH/dv)\lvert_{v=-v_{\text{ep3}}}$ decreases as $\mu_{2}$ is decreased. 
	Therefore, in the region $v<-v_{\text{ep3}}$ and $A>0$, the curve 
	$A = H(v)$ and line $A=-a_{\infty}(v +v_{\text{ep3}}) + A_{\text{ep3}}$ have 
	one common point when $\mu_{2}\in [\mu_2^{*}, \mu_{2}^{0}]$ 
	and 
	two common points when $\mu_{2}\in [\tilde{\mu}_{2}, \mu_{2}^{*})$ [Fig.\ref{fig:yz} (a)].

	Now, we define $X$ by 
	\begin{equation}\label{eqn:defXY}
		\begin{aligned}
			X &:= \{ (v, A, s) \ \lvert\ A \geq \widehat{F}(v), v\leq 0, s\geq 0\},\\
			\widehat{F}(v) &:= \left\{
			\begin{aligned}
				&-a_{\infty}(v +\hat{v}) + \widehat{A},  \text{\ for\ } v\leq -\hat{v},\\
				&\widehat{A},\hspace{2.4cm}  \text{\ for\ } -\hat{v}< v \leq 0,
			\end{aligned}
			\right.
		\end{aligned}
	\end{equation}
	where  
	$(\hat{v}, \widehat{A}) = (v_{\text{ep3}}, A_{\text{ep3}})$ for $\mu_{2}\in (\mu_{2}^{*}, \mu_{2}^{0}]$ and 
	$(\hat{v}, \widehat{A})=  (v_{X}, A_{X})$ for $\mu_{2}\in [\tilde{\mu}_{2}, \mu_{2}^{*}]$ 
	[Fig.\ref{fig:yz} (a)].  
	We will see in the next lemma that $\widehat{F}(v)$ is designed such that $X$ becomes an invariant set. 
	
	Because $\widehat{F}(v)$ is convex and 
	the slope of $A=G(v)$ at $v=-\infty$ is smaller than that at $-a_\infty$, the plane $A=\widehat{F}(v)$ does not have an intersection with 
	the plane $A = G(v)$ for $v \in (-\infty, -\hat{v})$ and the plane
	$A = H(v)$ for $v \in (-\hat{v}, \infty)$ [Fig.\ref{fig:yz} (a)]. 
	
	The boundaries of $X$ become
	\begin{equation*}
		\begin{aligned}
			&\partial^LX  := \left\{(v, A, s) \in X\ \lvert\ A = \widehat{F}(v), v \leq -\hat{v} \right\},\\
			&\partial^RX  := \left\{(v, A, s)\in X \ \lvert\ v = 0, A\geq \widehat{F}(v)\right\},\\
			&\partial^BX  := \left\{(v, A, s)\in X \ \lvert\  A = \widehat{F}(v), v > -\hat{v} \right\},\\
		\end{aligned}
	\end{equation*}
	where $L$, $R$, and $B$ mean left, right, and bottom, respectively [Fig.\ref{fig:yz} (b)]. 
	
	\begin{itemize}
		\item[(S6)] $M_3/M_2> \max\{\sqrt{3p_{21}/{p_{22}}}, 2p_{21}\sqrt{p_{11}p_{12}}/(p_{11}p_{22}-p_{12}p_{21})\}$. 
	\end{itemize}
	Under assumption (S6), we shall show in the next lemma that $X$ becomes an invariant set and that if the orbit enters into $X$, then $A$ has to diverge to $+\infty$. 
	
	\begin{figure}[htbp]
		\centering
		\includegraphics[width=12cm,clip]{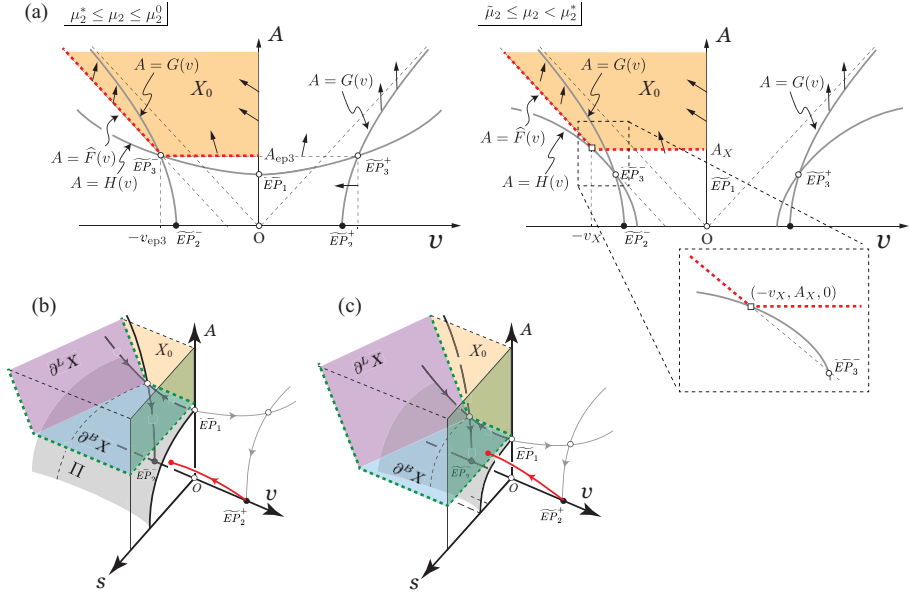}
		\caption{
			(a) The set $X$ is explicitly constructed as in \eqref{eqn:defXY}. 
			The subscript "0" denotes the boundary at $s=0$. 
			(b) $X$ is surrounded by four planes 
			of $X_0$, $\partial^{L}X$, $\partial^R X$, and $\partial^B X$ in $(v,A,s)$ space. 
			The surface $\Pi$ stands for the stable manifold of $\widetilde{\text{EP}}_{3}^{-}$, 
			which is a separator between preservation and annihilation. 
			For large $\mu_2$, the solution trajectory goes under $\Pi$ and reaches $\widetilde{\text{EP}}_{2}^{-}$ (preservation).
			(c) Decreasing $\mu_2$ to $\widetilde{\mu}_2$, the curve $A=H(v)$ 
			goes down. 
		}
		\label{fig:yz}
	\end{figure}

	The following lemma gives a sufficient condition for the divergence of $A$. 
	\begin{lemma}\label{lem:xinv}
		Let assumptions (S1)--(S6) hold. 
		If the solution enters $X$ in finite time, then 
		$A$ diverges. 
	\end{lemma}
	\begin{proof} 
		We investigate the flows at the boundary of $X$. \\
		\noindent\\
		Case 1: $\partial^{L}X$\\
		The normal vector at a point on $\partial ^{L} X$ is $(a_{\infty}, 1, 0)$ and 
		$(-\mu_{1} - p_{11}v^{2} + p_{12} A^{2})v>0$
		and $(-\mu_{2} - p_{21}v^{2} + p_{22} A^{2})A>0$ on $\partial^{L}X$. Thus, we obtain 
		\begin{align*}
			(a_{\infty}, 1, 0)\cdot (\dot{v}, \dot{A}, \dot{s}) 
			& = a_{\infty}[(-\mu_{1} - p_{11}v^{2} + p_{12} A^{2})v - M_{2} s]\\
			& \qquad + (-\mu_{2} - p_{21}v^{2} + p_{22} A^{2})A + M_{3}s\\
			&> M_{3} - a_{\infty}M_{2}>0,
		\end{align*}
		where we used (S6). This implies that the solution trajectory directs to the interior of $X$ on $\partial^{L}X$. \\\\
		Case 2: $\partial^{R}X$\\
		For a point on $\partial^{R}X$, because $v=0$, we have 
		\begin{align*}
			\dot{v} &= (-\mu_1-p_{11}v^2 + p_{12}A^{2})v-M_2s \\
			&= -M_{2}s<0. 
		\end{align*}
		\\
		Case 3: $\partial^{B}X$\\
		For $(v, A, s)$ on $\partial^BX$, because $(-\mu_2-p_{21}v^2+p_{22}A^2)A>0$ on $\partial^BX$, we have 
		\begin{equation*}
			\begin{aligned}
				(0, 1, 0)\cdot(\dot{v}, \dot{A}, \dot{s}) 
				&=  (-\mu_2-p_{21}v^2+p_{22}A^2)A + M_3 s\\
				&> M_3 s >0. 
			\end{aligned}
		\end{equation*} 
		Now, it is clear that $X$ is invariant. 
		
		Next, we show that $A$ diverges if the solution once enters $X$. 
		Suppose that the solution trajectory touches the boundary of $X$ at $t=t'<\infty$. 
		Because $s(t')>0$ and $\dot{A}(t') \geq  M_{3}s(t')$, there exist $\varepsilon>0$ and $\delta >0$
		such that $A(t)> A(t') +  \varepsilon$ for $t>t'+\delta$. 
		Because $A(t')>0$ and $(-\mu_{2} - p_{21}v^{2} + p_{22}A^{2})A > 0$ for $t>t'+\delta$, we have
		\begin{align*}
			\dot{A}(t) &= (-\mu_{2} - p_{21}v^{2}(t) + p_{22} A^{2}(t))A(t) + M_{3}s(t) \\
			&\geq [-\mu_{2} - p_{21}v^{2}(t) + p_{22}(A(t') + \varepsilon)^{2}](A(t')+\varepsilon)\\
			&>  [-\mu_{2} - p_{21}v^{2}(t) + p_{22}A(t')^{2}]A(t')+p_{22}\varepsilon^{2}A(t')\\
			&> p_{22}\varepsilon^{2}A(t').
		\end{align*}
		Thus, we find $A\to\infty$ as $t$ is increased. 
	\end{proof}

	\begin{lemma}\label{lem:vmin}
		If $\mu_{2}$ is in the preservation region, $v=O(\lvert\bm{\mu}\rvert^{1/2})$ and $s = O(\lvert\bm{\mu}\rvert)$ for all $t>0$. 
	\end{lemma}
	\begin{proof}
		Let us define the following set:
		$$
		Z:= \{(v, A, s)\ \lvert\ v> 0, s\geq 0\}.
		$$
		Firstly, we show that $A<\widehat{A}$ in $Z$. 
		Assume that the solution trajectory touches the plane $A=G(v)$  in $Z$. 
		Because $\dot{v}=-M_{2}s<0$ and $\dot{A}>0$ on $A=G(v)$ in $Z$, 
		the solution does not cross the plane $A=G(v)$ in $Z$. 
		Thus, the solution diverges in $Z$ or crosses 
		$\partial^{R}X$ and enters $X$. For both cases, $A$ diverges, which contradicts the condition that 
		$\mu_{2}$ is taken in the preservation regime. 
		Thus, $A<\widehat{A}$  for $0<t< \widetilde{T}$, where $\widetilde{T}$ is the time when  $v$ touches the plane $v=0$. 
		
		Secondly, we show that $v=O(\lvert\bm{\mu}\rvert^{1/2})$ for $t\geq \widetilde{T}$. 
		Because $\mu_{2}$ is taken in the preservation regime, $A<\widehat{A}$ for $t<T'$ ($T'$ is the time 
		when $v$ attains $-v_{\text{ep3}}$) from Lemma \ref{lem:xinv}. 
		It follows from Lemma \ref{lem:maxsadded} that $s(t) = O(\lvert\bm{\mu}\rvert)$ for $t\leq T'$, especially $s(T') =  O(\lvert\bm{\mu}\rvert)$. 
		We estimate the minimum value of $v$ for $t>T'$. Let us define $T''$ by the time $t$ when $v$ reaches $-v_{\text{ep3}}$ again, 
		where $T''=\infty$ if $v$ does not. For $t\in (T', T'']$, it holds that 
		\begin{align*}
			\dot{s} &< 2\alpha\left(-v_{\text{ep3}} + M_{1}s\right)s \\
			&< -\alpha v_{\text{ep3}}s, 
		\end{align*}
		where we used $s\ll \lvert v\rvert$ because $s=O(\lvert \bm{\mu}\rvert)$ and $\lvert v\rvert>v_{\text{ep3}} = O(\lvert\bm{\mu}\rvert^{1/2})$ for $t\in (T', T'']$. 
		Therefore, it follows that
		\begin{equation}\label{eqn:sesti}
			s(t) < s(T')\exp(-\alpha v_{\text{ep3}}t), \qquad t\in (T', T'']. 
		\end{equation}
		Because $\mu_{2}$ is taken in the preservation regime, from Lemma \ref{lem:xinv}, 
		the solution does not enter $X$ at $t=T'$. 
		That is, 
		\begin{align*}
			(-\mu_{1}-p_{11}v^{2}+p_{12}A^{2})v > 0
		\end{align*}
		should hold as long as $v<-v_{\text{ep3}}$ holds for $T'<t\leq T''$ from the definition of $X$ (see Fig.\ref{fig:yz}). 
		Because $s(T') = O(\lvert\bm{\mu}\rvert)$ (Lemma \ref{lem:maxsadded}) and $v_{\text{ep3}} = O(\lvert\bm{\mu}\rvert^{1/2})$, 
		by using \eqref{eqn:sesti}, we have the 
		following estimation for $t\in (T', T'']$:
		\begin{align*}
			v(t) - v(T')  &= \int_{T'}^{t}\dot{v} dt\\
			&= \int_{T'}^{t}(-\mu_{1} - p_{11}v^{2} + p_{12}A^{2})v - M_{2} s dt\\
			&> \int_{T'}^{t}- M_{2} s(T')\exp(-\alpha v_{\text{ep3}}t) dt\\
			&>  \int_{T'}^{\infty}- M_{2} s(T')\exp(-\alpha v_{\text{ep3}}t) dt
			= O(\lvert \bm{\mu}\rvert^{1/2}).
		\end{align*}
		That is,  it follows that 
		\begin{align*}
			v(t) > -v_{\text{ep3}} + O(\lvert\bm{\mu}\rvert^{1/2}). 
		\end{align*}
		Thus, $v(t) = O(\lvert\bm{\mu}\rvert^{1/2})$ for $t\in (T', T'']$. 
		
		If $v$ attains $-v_{\text{ep3}}$ again at $t=T'''>T''$, we can show that $v=O(\lvert\bm{\mu}\rvert^{1/2})$ for $t>T'''$ by using the 
		same argument. Therefore, $v=O(\lvert\bm{\mu}\rvert^{1/2})$ for $t\geq \widetilde{T}$, and we find 
		$v=O(\lvert\bm{\mu}\rvert^{1/2})$ for $t>0$.
	\end{proof}
	
	From these lemmas, we find that the variables $v$, $A$, and $s$ are sufficiently small as long as 
	$\mu_{2}$ belongs to the preservation regime. 
	\begin{lemma}\label{prop:3}
		Let assumptions (S1)--(S6) hold. For $\mu_{2}\in (\tilde{\mu}_{2}, \mu_{2}^{0})$, 
		the unstable manifold emanating from  $\widetilde{\mbox{EP}}_3^{-}$ consists of two parts: one-half of it converges to $\widetilde{\text{EP}}_{2}^{-}$ and the other half enters $X$. 
		That is, $\widetilde{\mbox{EP}}_3^{-}$ is a separator between preservation and annihilation. 
	\end{lemma}
	\begin{proof}
		The Jacobian matrix corresponding to $\widetilde{\text{EP}}_{3}^{-}$ is given by 
		\begin{equation*}
			J := 
			\begin{pmatrix}
				-2p_{11}v_{\text{ep3}}^{2} & -2p_{12}v_{\text{ep3}}A_{\text{ep3}} & -M_{2}\\
				2p_{21}v_{\text{ep3}}A_{\text{ep3}} & 2p_{22}A_{\text{ep3}}^{2} & M_{3}\\
				0 & 0 & -2\alpha v_{\text{ep3}}
			\end{pmatrix}
			.
		\end{equation*}	
		The matrix $J$ has eigenvalues 
		\begin{align*}
			\lambda_{1} &= p_{22}A_{\text{ep3}}^{2}-p_{11}v_{\text{ep3}}^{2} 
			+\sqrt{ (p_{22}A_{\text{ep3}}^{2}-p_{11}v_{\text{ep3}}^{2})^{2} + 4(p_{11}p_{22} - p_{12}p_{21})v_{\text{ep3}}^{2}A_{\text{ep3}}^{2} },\\
			\lambda_{2} &= p_{22}A_{\text{ep3}}^{2}-p_{11}v_{\text{ep3}}^{2} 
			-\sqrt{ (p_{22}A_{\text{ep3}}^{2}-p_{11}v_{\text{ep3}}^{2})^{2} + 4(p_{11}p_{22} - p_{12}p_{21})v_{\text{ep3}}^{2}A_{\text{ep3}}^{2} },\\
			\lambda_{3} &= -2\alpha v_{\text{ep3}}, 
		\end{align*}
		where the signs of the eigenvalues are 
		$(\lambda_{1}, \lambda_{2}, \lambda_{3}) = (+, -, -)$ 
		for $\mu_{2}\in (\tilde{\mu}_{2}, \mu_{2}^{0})$. 
		
		We show that unstable manifolds emanating from $\widetilde{\text{EP}}_{3}^{-}$ enter
		the region $Y_{0}$ for $\mu_{2}\in (\tilde{\mu}_{2}, \mu_{2}^{0})$, where $Y_{0}$ is a region surrounded by 
		the plane $A = G(v)$ and $v = -v_{\text{ep3}}$ in the $s= 0$ plane (Fig.\ref{fig:regionY}). That is, 
		one of the unstable manifolds of $\widetilde{\text{EP}}_{3}^{-}$ belongs to $X$, and the other converges to $\widetilde{\text{EP}}_{2}^{-}$. 
		
		For $\mu_{2}\in (\tilde{\mu}_{2}, \mu_{2}^{0})$, or $A_{\text{ep3}}\not= 0$, 
		the eigenvector corresponding to $\lambda_{1}$ 
		is given by $(1, m, 0)^{T}$, where 
		$$
		m = -\frac{p_{22}A_{\text{ep3}}^{2}+p_{11}v_{\text{ep3}}^{2}
			+\sqrt{ (p_{22}A_{\text{ep3}}^{2}-p_{11}v_{\text{ep3}}^{2})^{2} + 4p_{12}p_{21}v_{\text{ep3}}^{2}A_{\text{ep3}}^{2} }}{2p_{12}v_{\text{ep3}}A_{\text{ep3}}}
		<0.
		$$
		Because $-\mu_{1} - p_{11}v_{\text{ep3}}^{2} +  p_{12}A_{\text{ep3}}^{2} = 0$, it holds that
		
		\begin{equation}\label{eqn:fd}
			\left. \frac{dG}{dv}\right\lvert_{v=-v_{\text{ep3}}} = -\frac{p_{11}v_{\text{ep3}}}{\sqrt{p_{12}(\mu_{1} + p_{11}v_{\text{ep3}}^{2})}} 
			= 
			-\frac{p_{11}v_{\text{ep3}}}{p_{12}A_{\text{ep3}}} = -\frac{p_{11}}{p_{12}}\sigma,
		\end{equation}
		where $\sigma = v_{\text{ep3}}/A_{\text{ep3}}>0$. 
		In addition, 
		\begin{align}
			\notag
			m &= -\frac{p_{22}A_{\text{ep3}}^{2}+p_{11}v_{\text{ep3}}^{2}
				+\sqrt{ (p_{22}A_{\text{ep3}}^{2}+p_{11}v_{\text{ep3}}^{2})^{2} - 4p_{12}p_{21}v_{\text{ep3}}^{2}A_{\text{ep3}}^{2} }}{2p_{12}v_{\text{ep3}}A_{\text{ep3}}}\\
			& = -\frac{p_{22}+p_{11}\sigma^{2}+ \sqrt{(p_{22}+p_{11}\sigma^{2})^{2} - 4p_{12}p_{21}\sigma^{2}}}{2p_{12}\sigma}.
			\label{eqn:m}
		\end{align}
		By using assumption (S5) ($p_{11}p_{22} - p_{12}p_{21} > 0$), it holds that 
		\begin{align*}
			(p_{22}+p_{11}\sigma^{2})^{2} - 4p_{12}p_{21}\sigma^{2} &= p_{22}^{2} + 2p_{11}p_{22}\sigma^{2} + p_{11}^{2}\sigma^{4} - 4p_{12}p_{21}\sigma^{2}\\
			&> p_{22}^{2} + 2p_{11}p_{22}\sigma^{2} + p_{11}^{2}\sigma^{4} - 4p_{11}p_{22}\sigma^{2}\\
			&= (p_{22} - p_{11}\sigma^{2})^{2}. 
		\end{align*}
		Thus, 
		\begin{align}
			\notag
			p_{22}+p_{11}\sigma^{2}+ \sqrt{(p_{22}+p_{11}\sigma^{2})^{2} - 4p_{12}p_{21}\sigma^{2}}& >
			p_{22}+p_{11}\sigma^{2}+ \lvert p_{22} - p_{11}\sigma^{2}\rvert\\
			& = 
			\begin{cases}
				2p_{22}\quad &\text{if\ } p_{22} - p_{11}\sigma^{2}>0,\\
				2p_{11}\sigma^{2}\quad &\text{if\ } p_{22} - p_{11}\sigma^{2}\leq 0.\\
			\end{cases}
			\label{eqn:notecase}
		\end{align}
		From \eqref{eqn:fd},\eqref{eqn:m}, and \eqref{eqn:notecase}, we have the following:\\
		Case 1: If $p_{22} - p_{11}\sigma^{2}>0$, 
		\begin{align*}
			\left. \frac{dG}{dv}\right\lvert_{v=-v_{\text{ep3}}} - m > -\frac{p_{11}}{p_{12}}\sigma +\frac{p_{22}}{p_{12}\sigma} = \frac{1}{p_{12}\sigma}(p_{22}-p_{11}\sigma^{2} )>0.
		\end{align*}
		Case 2: If $p_{22} - p_{11}\sigma^{2}\leq 0$, 
		\begin{align*}
			\left. \frac{dG}{dv}\right\lvert_{v=-v_{\text{ep3}}}  - m > -\frac{p_{11}}{p_{12}}\sigma +\frac{p_{11}}{p_{12}}\sigma = 0. 
		\end{align*}
		Therefore, we obtain $(dG/dv)_{v=-v_{\text{ep3}}} - m > 0$. This means that the unstable manifolds emanating from $\widetilde{\text{EP}}_{3}^{-}$ enter
		$Y_{0}$. 
	\end{proof}
	
	\begin{figure}[htbp]
		\centering
		\includegraphics[width=12cm,clip]{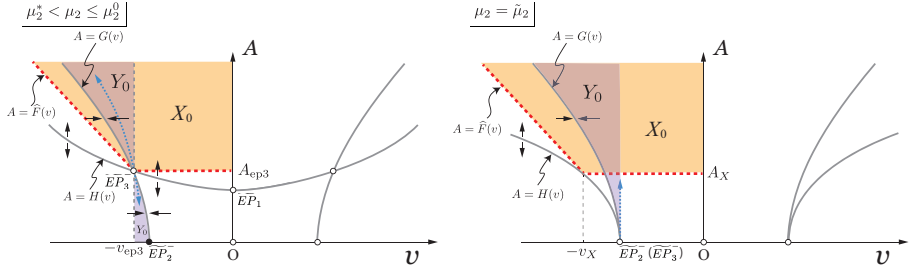}
		\caption{Region $Y_{0}:=\{ (v, A, 0) \ \lvert\ (A - G(v))(v+v_{\text{ep3}})\leq 0\}$.  The dashed arrows
			indicate the unstable manifolds of $\widetilde{\text{EP}}_{3}^{-}$.}
		\label{fig:regionY}
	\end{figure}
	
	\begin{theorem}[Traveling breather as a separator between preservation and annihilation]\label{thm:separator}
		Let assumptions (S1)--(S6) hold. If $\mu_1=-\mu_1^0$ and $\mu_2$ is decreased from $\mu_2^0$, then the solution 
		of (\ref{eqn:ode-2pulse-p}) converges to $\widetilde{\mbox{EP}}_3^{-}$ at some $\mu_2=\mu_2^c 
		\in (\tilde \mu_2, \mu_2^0)$, namely, the orbit lies on the stable manifold of the separator $\widetilde{\mbox{EP}}_3^{-}$ between preservation 
		and annihilation. For any orbit with initial data above (below) this manifold, its $A$ component diverges (it converges to $\widetilde{\mbox{EP}}_2^{-}$).
	\end{theorem}
	\begin{proof}
		Firstly, we show that annihilation occurs when $\mu_{2} = \tilde{\mu}_{2}$. 
		Let us define $X_{+}$ by
		\begin{equation*}
			X_{+} := \left\{(v, A, s)  \left\lvert \ A\geq - \frac{M_{3}}{M_{2}}(v - v_{\text{ep2}}), v>0, s\geq 0\right. \right\}.
		\end{equation*}
		We show that  the unstable manifold, emanating from $\widetilde{\text{EP}}_{2}^{+}$, corresponding to the positive eigenvalue 
		$\lambda = 2\alpha v_{\text{ep2}}$, 
		enters $X_{+}$. 
		The Jacobian matrix of $\widetilde{\text{EP}}_{2}^{+}$ is 
		\begin{equation*}
			J := 
			\begin{pmatrix}
				-2p_{11}v_{\text{ep2}}^{2} & 0 & -M_{2}\\
				0  & 0 & M_{3}\\
				0 & 0 & 2\alpha v_{\text{ep2}}
			\end{pmatrix}
			.
		\end{equation*}
		The eigenvector $(x, y, z)^{T}$ corresponding to the unstable eigenvalue 
		$\lambda=2\alpha v_{\text{ep2}}$ is 
		$$
		(x, y, z)^{T}= \left(-\frac{M_{2}}{2p_{11}v_{\text{ep2}}^{2}+ 2\alpha v_{\text{ep2}}}, 
		\frac{M_{3}}{2\alpha v_{\text{ep2}}}, 
		1
		\right)^{T}. 
		$$
		Because $v_{\text{ep2}}>0$, $p_{11}>0$, and $\alpha>0$, it follows that 
		\begin{equation*}
			\frac{y}{x} = 
			-\frac{M_{3}}{M_{2}} \frac{2p_{11}v_{\text{ep2}}^{2}+2\alpha v_{\text{ep2}}}{2\alpha v_{\text{ep2}}}< -\frac{M_{3}}{M_{2}}. 
		\end{equation*}
		By noting that $-M_{3}/M_{2}$ is the slope of the line $A= - \frac{M_{3}}{M_{2}}(v - v_{\text{ep2}})$, 
		we find that the solution trajectory emanating from $\widetilde{\text{EP}}_{2}^{+}$ enters $X_{+}$.

		The solution trajectory does not cross the plane 
		$A= - \frac{M_{3}}{M_{2}}(v - v_{\text{ep2}})$ in $X_{+}$ because 
		\begin{align*}
			\left(\frac{M_{3}}{M_{2}}, 1, 0\right)\cdot (\dot{v}, \dot{A}, \dot{s}) 
			& = \frac{M_{3}}{M_{2}}[(-\mu_{1} - p_{11}v^{2} + p_{12} A^{2})v - M_{2} s]\\
			& \qquad + (-\mu_{2} - p_{21}v^{2} + p_{22} A^{2})A + M_{3}s\\
			&> M_{3} - \frac{M_{3}}{M_{2}}M_{2}=0,
		\end{align*}
		where we used $(-\mu_{1} - p_{11}v^{2} + p_{12} A^{2})v>0$ and $(-\mu_{2} - p_{21}v^{2} + p_{22} A^{2})>0$ 
		on the plane $A= - \frac{M_{3}}{M_{2}}(v - v_{\text{ep2}})$ in $X_{+}$. 
		
		Now, we take $a_{\infty} = \min\{\sqrt{p_{11}/p_{12}}, M_{3}/M_{2}\}$. 
		We will show that
		$$
		\frac{M_{3}}{M_{2}}v_{\text{ep2}} > A_{X},
		$$
		where $A_{X}$ was defined in \eqref{eqn:Ax} and the left side is the $A$-intercept of the line 
		$A=-(M_{3}/M_{2})(v-v_{\text{ep2}})$ (Fig.\ref{fig:mu2=tildemu2}). 
		Namely, if $\frac{M_{3}}{M_{2}}v_{\text{ep2}} > A_{X}$ is satisfied, 
		the solution trajectory emanating from $\widetilde{\text{EP}}_{2}^{+}$ 
		enters $X$, which means that annihilation occurs for $\mu_{2} = \tilde{\mu}_{2}$ from Lemma \ref{lem:xinv}. Let us proceed to prove the above inequality. At $\mu_{2} = \tilde{\mu}_{2}$, 
		\begin{align*}
			v_{X} &= \sqrt{\frac{-\mu_{1}}{p_{11}}} \cdot \frac{a_{\infty}^{2}p_{22} +p_{21}}{a_{\infty}^{2}p_{22} -p_{21}},
			\\
			A_{X} &= \sqrt{\frac{-\mu_{1}}{p_{11}}} \cdot \frac{2a_{\infty} p_{21}}{a_{\infty}^{2}p_{22} - p_{21}}. 
		\end{align*}
		When $a_{\infty} = \sqrt{p_{11}/p_{12}}$, it holds from (S6) that 
		\begin{align*}
			\frac{M_{3}}{M_{2}}v_{\text{ep2}} - A_{X} = \sqrt{\frac{-\mu_{1}}{p_{11}}}
			\left(\frac{M_{3}}{M_{2}} - \frac{2p_{21} \sqrt{p_{11}p_{12}}}{p_{11}p_{22} -p_{12}p_{21}}\right)>0.
		\end{align*}
		When $a_{\infty} = M_{3}/M_{2}$, it holds from (S6) that 
		\begin{align*}
			\frac{M_{3}}{M_{2}}v_{\text{ep2}} - A_{X} = \sqrt{\frac{-\mu_{1}}{p_{11}}}
			\frac{\frac{M_{3}}{M_{2}}}{\left(\frac{M_{3}}{M_{2}}\right)^{2}p_{22} - p_{21}}
			\left\{\left(\frac{M_{3}}{M_{2}}\right)^{2}p_{22}-3p_{21}\right\}>0,
		\end{align*}
		where we note that 
		$$
		\left(\frac{M_{3}}{M_{2}}\right)^{2}p_{22} - p_{21} >\frac{1}{p_{12}}(p_{11}p_{22} - p_{12}p_{21}) > 0, 
		$$ 
		from (S5) and (S6). 
		
		Finally, we show that there exists $\mu_{2}^{c}\in (\tilde{\mu}_{2}, \mu_{2}^{0})$ such that 
		the solution trajectory converges to the separator between preservation and annihilation at $\mu_{2}=\mu_{2}^{c}$. 
		The solution trajectory converges to $\widetilde{\text{EP}}_{2}^{-}$ 
		at $\mu_{2}= \mu_{2}^{0}$ by definition (preservation), while at the same time, it follows from the above arguments that the solution 
		trajectory enters $X$ when $\mu_{2}= \tilde\mu_{2}$. Then, we conclude from the continuous dependency of the trajectory on the parameters that there exists at least one $\mu_{2}^{c}\in (\tilde{\mu}_{2}, \mu_{2}^{0})$ such that 
		the solution lies on the stable manifold of $\widetilde{\text{EP}}_{3}^{+}$ at some $\mu_{2} = \mu_{2}^{c}$ (Fig.\ref{fig:z}).  
		In other words, the stable manifold of $\widetilde{\text{EP}}_{3}^{+}$ is 
		a separator between preservation and annihilation. 
	\end{proof}
	
	\begin{remark}
		The uniqueness of the critical parameter $\mu_{2}^{c}$ is not known in Theorem \ref{thm:separator}; however, if it holds and the transversal property is satisfied at $\mu_{2}^{c}$, then preservation and annihilation regimes are completely separated by $\mu_{2}^{c}$. 
	\end{remark}

	\begin{figure}[htbp]
		\centering
		\includegraphics[width=8cm,clip]{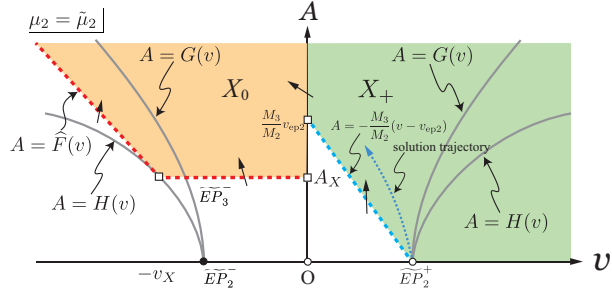}
		\caption{Schematic figure of the solution trajectory when $\mu_{2} = \tilde{\mu}_{2}$. }
		\label{fig:mu2=tildemu2}
	\end{figure}
	
	\begin{figure}[htbp]
		\centering
		\includegraphics[width=6cm,clip]{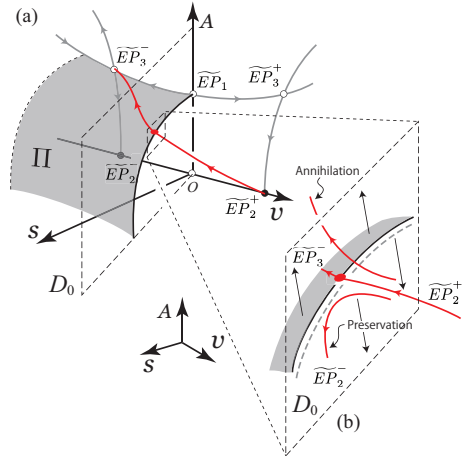}
		\caption{
			(a) Schematic picture of a solution trajectory coming from $\widetilde{\mbox{EP}}_2^+$
			in $(v,A, s)$ space at the critical value $\mu_2^c$. 
			The solution trajectory goes towards $\widetilde{\mbox{EP}}_3^-$ 
			along the gray colored surface of $\Pi$. 
			(b) Depending on which side of $\Pi$ each orbit, 
			the orbits are sorted into the annihilation or preservation
			behaviors for $\mu_2 < \mu_2^c$ or $\mu_2 > \mu_2^c$, respectively. 
		}
		\label{fig:z}
	\end{figure}

	\subsection{Comparison with PDE dynamics}\label{sec:comparison}
	Let us confirm numerically that the reduced system (\ref{eqn:ode-2pulse-p}) gives us qualitatively the same phase diagram as that of the PDE model (see Fig.\ref{fig:pd-ode}(a)). 
	In numerical simulations, the initial data $(v(0),A(0), s(0))$ is taken near $\widetilde{\mbox{EP}}_2^+$ if it
	exists and is taken near $(0,0,0)$ if it does not (see Appendix A for details).
	Recall that the criterion for the occurrence of ``annihilation'' is that the component $A$ diverges, and ``preservation'' means that the solution converges 
	to $\widetilde{\mbox{EP}}_2^-$. The background state shows the
	non-existence of both standing and traveling pulses. 
	The ODE phase diagram Fig.\ref{fig:pd-ode}(b) is obtained based on this classification. Fig.\ref{fig:numerics2} shows typical orbital behaviors depending on the parameters. It is clear from previous discussions that the Hopf and drift lines separate background, standing, and preservation regions counterclockwise like I $\rightarrow$ II $\rightarrow$ III, similar to the PDE diagram.  
	The transition from preservation to annihilation occurs as $\mu_{1}$ is decreased and the lower boundary of the annihilation regime is $T_2$. The most subtle issue is to examine whether the tongue of the annihilation regime touches the DH point or not, which is unclear due to the limitation of numerical discretization.
	
	The following analytical result helps us to resolve the issue, which is obtained directly from Theorem \ref{thm:separator}. 
	
	\begin{corollary}\label{cor:weak}
		Let assumptions (S1)--(S6) hold.  Then, we have $\mu_2^{c}\to 0$ as $\mu_{1}^{0}\to 0$.
	\end{corollary}
	\begin{proof} Because the preservation dynamics is observed for $(\mu_{1}, \mu_{2}) =(\mu_{1}^{0}, \mu_{2}^{0})$ ($\mu_{1}^{0}<0, \mu_{2}^{0}>0$) 
		arbitrarily close to the origin, 
		the critical $\mu_{2}^{c}$ can be taken arbitrarily close to $0$ from Theorem \ref{thm:separator}. 
	\end{proof}

	\begin{figure}[htbp]
		\centering
		\includegraphics[width=11cm,clip]{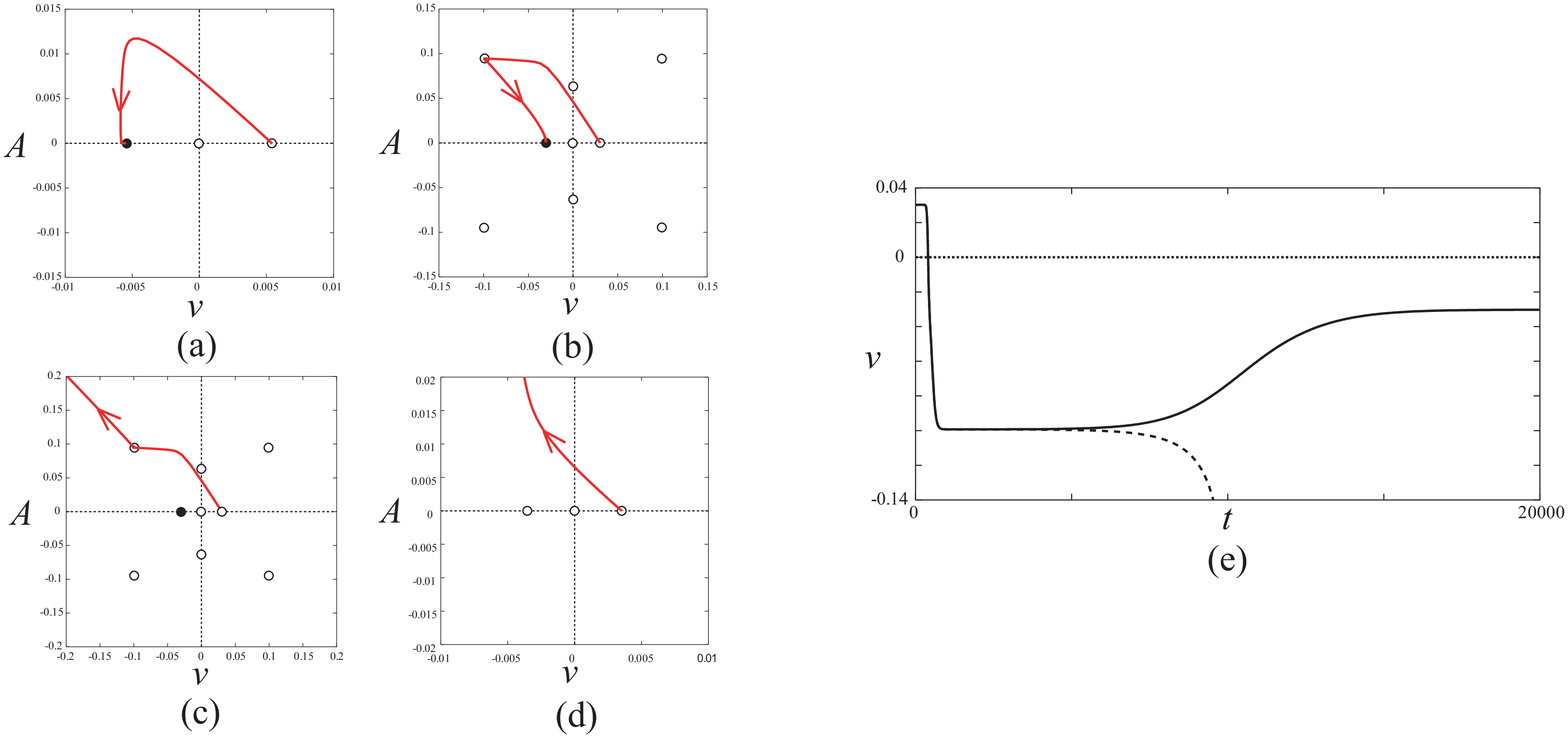}
		\caption{
			(a)-(d) Orbit of (\ref{eqn:ode-2pulse-p}) projected to the $v$-$A$
			plane for $h(0)=500$. 
			The parameters $p_{ij}$ are the same as those given in Fig.\ref{fig:pd-ode}. 
			Filled and open circles indicate 
			stable and unstable equilibrium points, respectively. 
			The parameter values are changed by taking control parameter
			$\theta$ with 
			$(\mu_1,\mu_2) = (0.001 \cos\theta, 0.001 \sin\theta)$.
			(a) $\theta=1.8$ [rad] [region (iii)]. 
			(b) $\theta=2.7283$ [rad] [region (iii)]. 
			(c) $\theta=2.7284$ [rad] [region (iii)]. 
			(d)$\theta=4.7$ [rad] [region (iv)]. 
			(e) $t$-$v$ plot of (b) (solid line)  and (c) (dashed line). 
		}
		\label{fig:numerics2}
	\end{figure}
	
	Theorem \ref{thm:separator} indicates that the solution trajectory approaches UTB at $\mu_{2}^{c}$ after weakly symmetric collision. 
	In addition, from Corollary \ref{cor:weak}, the tip of the annihilation region touches the DH point, and WIIA occurs for arbitrarily slow and weak collisions. 
	This suggests that annihilation occurs for any arbitrary slowly colliding pulses of the PDE system (\ref{eqn:gd1}), which is not an easy task to realize by numerical simulations. However, if we choose a set of parameter values carefully in the annihilation regime (but close to the border between annihilation and preservation), we observe the behavior depicted in Fig.\ref{fig:birdview-codim3}. 
	In fact, it does not look like a collision because they are well-separated even at the nearest point. They rebound after that and start to oscillate up and down before collapsing to the background state. The last stage indicates that the PDE solution actually behaves like the separator UTB.
	
	At first sight, this looks contradictory to the results in
	\cite{ei-mimura-nagayama}, in which they showed that preservation occurs when two pulses collide weakly enough. However, in our case, 
	the criterion of the fate of the dynamics
	is not strength of interaction but the size of 
	basin of attraction of preservation, which is essentially a
	different viewpoint from that in \cite{ei-mimura-nagayama}. 
	In fact, as shown in Fig.\ref{fig:birdview-codim3}(b), the minimum distance of two pulses in the annihilation regime is 
	larger than that for preservation, which suggests that the strength of 
	interaction is not a main factor causing annihilation.
	
	Overall, both the PDE and reduced ODEs present qualitatively the same phase diagrams, as depicted in Fig.\ref{fig:pd-ode}. The reduced ODEs clearly show the location of transitions and how they happen there; in particular, the transition from preservation to annihilation is highlighted in the previous section. 
	
	\section{Conclusion and outlook}
	
	\begin{itemize}
		
		\item
		{\bf{Codimension 3 singularity}}:\\
		We considered the hidden instability of DH type in this paper. It is a combination of drift and Hopf instabilities. The meaning of ''hidden" is that it is not visible when the pattern is isolated. It comes out when the pattern starts to interact with other traveling patterns or defects. A nontrivial characteristic is that even if they interact very weakly, they experience a large deformation and annihilate each other provided that all the parameters are chosen close to the DH point. In a sense, the pattern is very fragile when it includes such instabilities inside. There is one thing we have to pay slightly more attention to understand the dynamics of annihilation: the existence of a saddle-node bifurcation of the standing pulse branch, as shown in Fig.\ref{fig:codim3-bifdiag-gam8.0}(a)(b). In fact, annihilation heavily depends on the saddle-node point that determines the destinations of the unstable manifolds of UTP and UTB (i.e., homogeneous background state). In this paper, we assumed that the destinations of UTP and UTB are the background state and did not consider the effect explicitly coming from the saddle-node bifurcation. It would be reasonable to consider the annihilation in the framework of codimension 3 singularity, i.e., DH point + saddle--node bifurcation; however, we leave this as a future work.\\
		
		\item 
		{\bf{Global structure of PDE branches}}:\\
		It is an interesting future task to explore the global behaviors of PDE solution branches as schematically depicted in Fig.\ref{fig:codim3-bifdiag-gam8.0}(c)(d). In particular, unstable traveling breathers UTB play a key role in understanding the transition from preservation to annihilation; therefore, more precise bifurcation structure of secondary branches may help us to explore the behavior of the basin structure of traveling pulses.\\
		
		\item 
		{\bf{2-dimensional spot annihilation}}:\\
		We showed the annihilation of 2-dimensional spots in Section 1 for the two representative three-component reaction diffusion systems. It is possible to derive a similar reduced finite-dimensional system associated with the 2-dimensional case, as we did in Section 3. We believe that the same mechanism of annihilation as the 1-dimensional case works also for the 2-dimensional case; however, more careful analysis is needed due to the increase of unknowns, this remains as a future work. \\
		
		\item 
		{\bf{Oblique collision and coherent motion of many traveling spots}}:\\
		The search for a coherent motion of many traveling spots is an interesting open question.
		It is numerically observed that if each traveling spot is robust (i.e., not close to hidden instabilities), then repulsive interaction drives them to the marching mode (i.e., travel in the same direction with same velocity) in a bounded domain with periodic boundary condition. However, for the spots close to DH singularity, they may annihilate upon collision when they are close to a head-on collision. It is not known how the angle between two traveling spots affect the outcome after collision, what is the asymptotic behavior as time goes on, and how the final state depends on the density of spots.\\
		
		\item
		{\bf{Heterogeneous media}}
		Defects or heterogeneities in the media is another type of external perturbation to propagating spots and pulses, which may trigger pinning to the defects, as was shown in \cite{heij-nishiura-2011,nishi-nishiura-teramoto-2018,heij-nishiura-2019,heij-teramoto-2020}. Suppose a propagating pulse or spot has a hidden instability of DH type; then, it is natural to observe a similar annihilation after collision against a defect. It is confirmed numerically that a 1-dimensional traveling pulse experiences a similar annihilation to that depicted in Fig.\ref{fig:birdview-codim3}. For higher-dimensional cases, we conjecture that the same type of annihilation occurs, but the details are left as a future work.

	\end{itemize}

\appendix
\section{Phase diagram for the reduced ODE}
In the numerical simulations for \eqref{eqn:ode-2pulse-p}, 
the initial data $(s(0),v(0),A(0))$ is taken near $\widetilde{EP}_2^+$ if 
exists and is taken near $(0,0,0)$ if does not:
\begin{equation*}
 (v(0), A(0), s(0)) = \left\{
\begin{aligned}
 &(\sqrt{-\mu_1/p_{11}}+0.001, 0.001, \exp(-50)),\quad \text{for}\quad \mu_1<0,\\
 &(0.001, 0.001, \exp(-50)), \qquad\qquad\qquad\qquad \text{otherwise}.
\end{aligned}
\right.
\end{equation*}
Figure \ref{fig:pd-ode}(b) shows the phase diagram. 
Here, in convenience,  we suppose the criterion for the occurrence of ``annihilation'' is that 
$\lvert A\rvert$  reaches $1$. When the solution $(v,A,s)$ converges to $\widetilde{\text{EP}}_0$, 
$\widetilde{EP}_2^-$, 
the dynamics is called ``standing'' and ``preservation", respectively. 
The region `` background'' corresponds to the parameter region where no stable solution exists. 

\section{Proof of Proposition \ref{thm:ode-2pulse} and Theorem \ref{cor:1}}\label{appen:A}
\subsection{Proof of Proposition \ref{thm:ode-2pulse}}\label{appen:B1}
Let 
$\bmV:=\sum_{j=1,2}S_j(x)+q_j\psi_j(x)+(r_j\xi_j(x)+\text{c.c.})+\zeta_j(x)$. 
Since $\Xi'(p_1)=-\Xi(p_1)(\partial/\partial x)$, by substituting
(\ref{eqn:1}) into the left hand of (\ref{eqn:0}), we have
\begin{equation}\label{eqn:lhs}
\begin{aligned}
 (I+E_{n,n}(\eta_2))\bmu_t &=
  (I+E_{n,n}(\eta_2)) \biggl[\dot{p}_1\Xi'(p_1)(\bmV+\Theta(h)\bmw)\\
&\quad +\Xi(p_1)
\left(\frac{\partial}{\partial(h,\bmq,\bmr)}(\bmV+\Theta(h)\bmw)(\dot{h},\dot{\bmq},\dot{\bmr})
 +\Theta(h)\bmw_t\right)\biggr]\\
 &= 
  \Xi(p_1)\biggl[(I+E_{n,n}(\eta_2))\biggl(
 -\dot{p}_1\frac{\partial}{\partial x}(\bmV+\Theta(h)\bmw)\\
&\quad +\frac{\partial}{\partial(h,\bmq,\bmr)}(\bmV+\Theta(h)\bmw)(\dot{h},\dot{\bmq},\dot{\bmr})
 +\Theta(h)\bmw_t\biggr)\biggr].
\end{aligned}
\end{equation}

By expanding the right hand side of (\ref{eqn:lhs}), we have
\begin{equation*}
 \begin{aligned}
 &(I+E_{n,n}(\eta_2))\bmu_t =\Xi(p_1)(I+E_{n,n}(\eta_2))\\
&\biggl\{
  \sum_{j=1,2}\bigl[ -\dot{p}_j(\partial_x S_j + q_j\partial_x\psi_j+(r_j\partial_x\xi_j+\text{c.c.} +
  \partial_x \zeta_j + \Theta(h)\bmw_x)) \\
  &\quad + \dot{q}_j (\psi_j+\partial_q \zeta_j) 
  +
  (\dot{r}_j(\xi_j+\partial_r \zeta_j) +\text{c.c.})\bigr] + \Theta(h)\bmw_t
\biggr\}.
\end{aligned}
\end{equation*}

On the other hand, by substituting
(\ref{eqn:1}) into the right hand of (\ref{eqn:0}), we have
\begin{equation}\label{eqn:rhs}
 \begin{aligned}
 {\cal L}(\bmu)+\eta_1 g(\bmu)&= \Xi(p_1)[{\cal L}(\bmV+\Theta(h)\bmw)+\eta_1g(\bmV+\Theta(h)\bmw)] \\
&
 =\Xi(p_1)[\sum_{j=1,2}({\cal L}'(S_j)\bmw-q_j\partial_x S_j+(i\omega r_j\xi_j+\text{c.c.})
  +q_j^2(\alpha_{2000}\xi_j-\Pi_{2000})\\
 & +(r_j^2(\alpha_{0200}\xi_j-\Pi_{0200})
 +q_jr_j(\alpha_{1100}\psi_j+\alpha'_{1100}\partial_x S_j-\Pi_{1100})+\text{c.c.})\\
 & +\lvert r_j\rvert^2(\alpha_{0110}\xi_j-\Pi_{0110})
 +\eta_1(\alpha_{0001}\xi_j-\Pi_{0001}))\\
 & 
  +q_j^3({\cal F}''(S_j)\psi_j\cdot\zeta_{2000,j}+\frac16{\cal F}'''(S_j)(\psi_j)^3) 
+q_j\lvert r_j\rvert^2({\cal F}''(S_j)\xi_j\cdot\zeta_{1010,j} \\
  & +{\cal F}''(S_j)\bar\xi_j\cdot\zeta_{1100,j}+
{\cal F}'''(S_j)\psi_j\cdot\xi_j\cdot\bar{\xi_j})\\
&
  +\eta_1g_1(S_j)
  +q_j\eta_1({\cal F}''(S_j)\psi_j\cdot\zeta_{0001,j}+g_1'(S_j)\psi_j)\\
&
  +\{q_j^2r_j(\frac12{\cal F}'''(S_j)(\psi_j)^2\cdot\xi_j+{\cal F}''(S_j)\psi_j\cdot\zeta_{1100,j}
+{\cal F}''(S_j)\xi_j\cdot\zeta_{2000,j})\\
  &
  +r_j^3({\cal F}''(S_j)\xi_j\cdot\zeta_{0200,j}
  +\frac16{\cal F}'''(S_j)(\xi_j)^3) 
  +q_jr_j^2({\cal F}''(S_j)\psi_j\cdot\zeta_{0200,j}\\
&+{\cal F}''(S_j)\xi_j\cdot\zeta_{1100,j}
+\frac12{\cal F}'''(S_j)\psi_j\cdot(\xi_j)^2)\\
  &
  +r_f\lvert r_j\rvert ^2({\cal F}''(S_j)\xi_j\cdot\zeta_{0110,j}+{\cal F}''(S_j)\bar\xi_j\cdot\zeta_{0200,j}+
\frac12{\cal F}'''(S_j)(\xi_j)^2\cdot\bar{\xi_j})\\
 & 
  +r_j\eta_1(g_1'(S_j)\xi_j+{\cal F}''(S_j)\xi_j\cdot\zeta_{0001,j})+\text{c.c.}\}\\
 &
 + q_j{\cal F}''(S_j)\xi_j\cdot\bmw + (r_j{\cal F}''(S_j)\xi_j\cdot\bmw+\text{c.c.})
 +O(\lvert \eta\rvert^{2} + \lvert \bmq\rvert ^{4} + \lvert\bmr\rvert^{4})].\\
 \end{aligned}
\end{equation}

By operating $Q$ on (\ref{eqn:lhs}) and (\ref{eqn:rhs}), we obtain the
following estimation:
\begin{align*}
 \dot{p}_j&= O(\lvert\bmq\rvert(1+\lvert\bmr\rvert+\lvert\bmeta\rvert)+\delta),\\
 \dot{q}_j&= O(\lvert\bmq\rvert(\lvert \bmr\rvert+\lvert\bmeta\rvert)+\delta),\\
 \dot{r}_j &= i\omega r_j + O(\lvert\bmq\rvert^2+\lvert\bmr\rvert^2+\lvert\bmeta\rvert), \\
\end{align*}
where we have used $\langle {\cal F}''(S_j)\psi_j\cdot\bmw, \phi^*\rangle_{L^2}$, 
$\langle {\cal F}''(S_j)\xi_j\cdot\bmw, \xi^*\rangle_{L^2}$, 
$\langle \bmw_x, \phi^*\rangle_{L^2}$ and $\langle \bmw_x, \xi^*\rangle_{L^2}$ are
  $O(\lvert\bmq\rvert^3+\lvert\bmr\rvert^3 +\lvert\bmeta\rvert^{3/2})$, which are derived in a similar way 
 those given in the proof of Theorem 4.1 in \cite{ei-mimura-nagayama}. 
By operating $Q$ on (\ref{eqn:lhs}) and (\ref{eqn:rhs}) again with the 
estimations above, we obtain (\ref{eqn:ode-final}).

\subsection{Proof of Theorem \ref{cor:1}}\label{sec:appenA2}
In order to vanish quadratic terms with respect to $q_j$ and $r_j$, 
we perform a smooth invertible transformation:
\begin{equation}\label{eqn:henkan1}
 \begin{aligned}
 v_j &= q_j + V_{1100}q_j r_j
  +V_{1010}q_j\bar{r}_j+V_{1200}q_jr_j^2
 +V_{1020}q\bar{r}^2,\\
 w_j &= r_j + W_{0001}\eta_1+ W_{2000}q_j^2
 + W_{0200}r_j^2 +W_{0020}\bar{r}_j^2
  +W_{0110}\lvert r_j\rvert^2\\
 &+W_{0011}\eta_1\bar{r}_j+
   W_{2010}q_j^2\bar{r}_j
 +W_{0120}\bar{r}_j\lvert r_j\rvert ^2+W_{0300}r_j^3+W_{0030}\bar{r}_j^3,
  \end{aligned}
\end{equation}
where 
\begin{equation*}
\begin{aligned}
 &
 V_{1100} = -\frac{g_{1100}}{i\omega},
 V_{1010}=\frac{g_{1010}}{i\omega},\\
 &
  V_{1200} =
 -\frac{g_{1200}+g_{1100}V_{1100}+h_{0200}V_{1100}+h_{0200}V_{1010}}{2i\omega},\\
 & 
  V_{1020} =
 \frac{g_{1020}+g_{1010}V_{1010}+h_{0020}V_{1100}+h_{0020}V_{1010}}{2i\omega},\\
 &
  W_{0001}= \frac{h_{0001}}{i\omega},\\
 &
  W_{0011}= \frac{h_{0011}+2W_{0020}\bar
 h_{0001}+W_{0110}h_{0001}}{2i\omega},\\ 
 & 
  W_{2000} =\frac{h_{2000}}{i\omega},
  W_{0200}=-\frac{h_{0200}}{i\omega},
  W_{0020} = \frac{h_{0020}}{3i\omega},
  W_{0110}=\frac{h_{0110}}{i\omega},\\
&
  W_{2010}=\frac{h_{2010}+2W_{2000}g_{1010}+2W_{0020}h_{2000}
  +W_{0110}{h_{2000}}}{2i\omega},\\
 &
  W_{0120}= \frac{h_{0120}+2W_{0200}h_{0020}+2W_{0020}\bar h_{0110}
  +W_{0110}h_{0110}+W_{0110}h_{0020}}{2i\omega},\\
 &
  W_{0300} =
  -\frac{h_{0300}+2W_{0200}h_{0200}+W_{0110}h_{0200}}{2i\omega},\\
 &
  W_{0030} = \frac{h_{0030}+2W_{0020}h_{0020}+W_{0110}h_{0020}}{4i\omega}.
\end{aligned}
\end{equation*}
Differetiating both side of (\ref{eqn:henkan1}) and substituting into
(\ref{eqn:ode-final}), we have 
\begin{equation*}
 \begin{aligned}
 \dot{v}_j&=G_{1001}\eta_1 v+G_{1002}\eta_2 v_j+
G_{3000}v_j^3+G_{1110}v_j\lvert w_j\rvert^2 + (-1)^{j} M_2 s + \text{h.o.t.}\\
 \dot{w}_j&= i\omega w_j+H_{0101} \eta_1 w_j
+H_{2100} v_j^2w_j
  +H_{0210}w_j\lvert w_j\rvert^2+ M_3 s+\text{h.o.t.}
  \end{aligned}
\end{equation*}
where
\begin{equation*}
 \begin{aligned}
 &
  G_{1001}=g_{1001}+V_{1100}h_{0001}+V_{1010}\bar h_{0001},\\
 &
  G_{1002}=g_{1002},\\
 &
  G_{3000}=g_{3000}+V_{1100}h_{2000}+V_{1010}\bar h_{2000},\\
 &
  G_{1110}=g_{1110}+V_{1100}g_{1010}+V_{1100} h_{0110}
  +V_{1010}g_{1100}+V_{1010}\bar{h}_{0110},\\
 &
  H_{0101}=h_{0101}+2W_{0200}h_{0001}+W_{0110}\bar h_{0001},\\
 &
  H_{2100}=h_{2100}+2W_{2000}g_{1100}
  +2W_{0200}h_{2000}
 +W_{0110}\bar h_{2000},\\
 &
  H_{0210}=h_{0210}+2W_{0200}h_{0110}
  +2W_{0020} h_{0200}
  +W_{0110}h_{0200}
  +W_{0110}\bar h_{0110}.\\
    \end{aligned}
\end{equation*}

In order to derive an amplitude equation (\ref{eqn:ode-2pulse}), 
we change $w_j$ as  $w_j = A_je^{i\varphi_j} (A_j, \varphi_j\in \mathbb{R})$.

\section{The derivation of (\ref{eqn:cmf})}  \label{sec:appendc}
We consider the following equation around $\mu_{1} = 0$: 
\begin{equation}\label{eqn:odeappend}
 \begin{aligned}
 \dot{s} &= 2\alpha  (v+M_1s)s,\\
 \dot{v} &= (-\mu_1-p_{11}v^2+p_{12}A^2)v
  -M_2s, \\
 \dot{A} &= (-\mu_2-p_{21}v^2+p_{22}A^2)A
  +M_3s.
 \end{aligned}
\end{equation}
The linearized matrix of (\ref{eqn:odeappend}) is 
\begin{equation*}
\begin{pmatrix}
 0 & 0 & 0 \\
 -M_2 & 0 & 0 \\
 M_3 & 0 & -\mu_2
\end{pmatrix}
\end{equation*}
and the eigenvalues are $\lambda = 0, 0, -\mu_2$. 
The eigenvector corresponding to  $\lambda=0$ is given by the 
following equation:
\begin{equation*}
\begin{pmatrix}
 0 & 0 & 0 \\
 -M_2 & 0 & 0 \\
 M_3 & 0 & -\mu_2
\end{pmatrix}
\begin{pmatrix}
 x \\
 y \\
 z
\end{pmatrix}
=
\begin{pmatrix}
 0 \\
 0 \\
 0
\end{pmatrix}
.
\end{equation*}
Thus, we have $(x,y,z) = (0, M_2\mu_2,0)$. 
Another eigenvalue for $\lambda=0$ is given by the following equation:
\begin{equation*}
\begin{pmatrix}
 0 & 0 & 0 \\
 -M_2 & 0 & 0 \\
 M_3 & 0 & -\mu_2
\end{pmatrix}
\begin{pmatrix}
 x \\
 y \\
 z
\end{pmatrix}
=-
\begin{pmatrix}
 0 \\
 M_2\mu_2 \\
 0
\end{pmatrix}
.
\end{equation*}
Thus, we have $(x,y,z) = (\mu_2, 0, M_3)$. 
The eigenvector corresponding to $\lambda=-\mu_2$ is given by 
\begin{equation*}
\begin{pmatrix}
 -\mu_2 & 0 & 0 \\
 -M_2 & -\mu_2 & 0 \\
 M_3 & 0 & 0
\end{pmatrix}
\begin{pmatrix}
 x \\
 y \\
 z
\end{pmatrix}
=-
\begin{pmatrix}
 0 \\
 0 \\
 0
\end{pmatrix}
.
\end{equation*}
Thus we have $(x,y,z) = (0, 0, 1)$. Therefore the transformation matrix
$P$ is given by 
\begin{equation*}
P=
\begin{pmatrix}
 0 & \mu_2 & 0 \\
 M_2\mu_2 & 0 & 0 \\
 0 & M_3 & 1
\end{pmatrix}
,\qquad P^{-1}=
\begin{pmatrix}
 0 & \frac{1}{M_2\mu_2} & 0 \\
 \frac{1}{\mu_2} & 0 & 0 \\
 -\frac{M_3}{\mu_2} & 0 & 1
\end{pmatrix}
.
\end{equation*}
Therefore it holds that 
\begin{equation*}
\begin{pmatrix}
 0 & \frac{1}{M_2\mu_2} & 0 \\
\frac{1}{\mu_2} & 0 & 0 \\
 -\frac{M_3}{\mu_2} & 0 & 1
\end{pmatrix}
\begin{pmatrix}
 0 & 0 & 0 \\
 -M_2 & 0 & 0 \\
 M_3 & 0 & -\mu_2
\end{pmatrix}
\begin{pmatrix}
 0 & \mu_2 & 0 \\
 M_2\mu_2 & 0 & 0 \\
 0 & M_3 & 1
\end{pmatrix}
=
\begin{pmatrix}
 0 & -1 & 0 \\
 0 & 0 & 0 \\
 0 & 0 & -\mu_2
\end{pmatrix}
.
\end{equation*}
Here we take
\begin{equation*}
P^{-1}
\begin{pmatrix}
 s \\
 v\\
 A\\
\end{pmatrix}
=
\begin{pmatrix}
x\\
y\\
z 
\end{pmatrix}
.
\end{equation*}
That is, $s=\mu_2y$, $v=M_2\mu_2x$, $A=M_3y+z$. Substituting these
into (\ref{eqn:ode-2pulse-p}), we have
\begin{equation*}
 \begin{aligned}
  &\dot{x} = -y - \mu_1 x + O(\lvert x^3\rvert+\lvert xy^2\rvert+\lvert xyz\rvert+\lvert xz^2\rvert),\\
  &\dot{y}= 2\alpha M_2\mu_2 xy,\\
  &\dot{z} = -\mu_2 z - 2\alpha M_2M_3\mu_2 x y 
  + O(\lvert x^2y\rvert+\lvert x^2z\rvert +\lvert y^3\rvert+\lvert y^2z\rvert+\lvert yz^2\rvert+\lvert z\rvert^3). 
\end{aligned}
\end{equation*}
In order to apply the center manifold theory, we take $\mu_1$ as a
function of $t$ and consider the follwoing equations 
\begin{equation}\label{eqn:center1}
 \begin{aligned}
  &\dot{x} = -y - \mu_1 x + O(\lvert x^3\rvert +\lvert xy^2\rvert+\lvert xyz\rvert +\lvert xz^2\rvert),\\
  &\dot{y}= 2\alpha M_2\mu_2 xy,\\
  &\dot{z} = -\mu_2 z - 2\alpha M_2M_3\mu_2 x y 
  + O(\lvert x^2y\rvert+\lvert x^2z\rvert + \lvert y^3\rvert+\lvert y^2z\rvert+\lvert yz^2\lvert +\lvert z\rvert^3),\\
  &\dot{\mu}_1= 0.
\end{aligned}
\end{equation}
Since $\mu_2>0$, the variable $z(t)$ on the center manifold can be
expressed by the quadratic equation:
\begin{equation}\label{eqn:anzatsz}
z=b_1x^2+b_2 y^2+b_3 \mu_1^2+b_4 xy+b_5\mu_1y 
+ b_6\mu_1x, 
\end{equation}
where $b_j\in \mathbb{R}$. By differentiate both sides of
(\ref{eqn:anzatsz}) using (\ref{eqn:center1}), we have 
\begin{equation}\label{eqn:zhikaku1}
 \dot{z} = -2b_1xy-b_4 y^2-b_6 \mu_1 y
\end{equation}
On the other hand, putting (\ref{eqn:anzatsz}) in the $z$ equation 
of (\ref{eqn:center1}), whe have 
\begin{equation}\label{eqn:zhikaku2}
 \dot{z}= -\mu_2(b_1x^2+b_2 y^2+b_3 \mu_1^2+b_4 xy+b_5\mu_1y)-2\alpha
  M_2M_3\mu_2 xy
\end{equation}
Comparing the second order terms between (\ref{eqn:zhikaku1}) and
(\ref{eqn:zhikaku2}), we have $b_1=b_3=b_5=b_6=0$, $b_2=-2\alpha M_2M_3/\mu_2$, 
$b_4=-2\alpha M_2M_3$. 
Therefore, 
\begin{equation*}
 z = -\frac{2\alpha M_2M_3}{\mu_2}y^2-2\alpha M_2M_3 xy. 
\end{equation*}
Thus, we obtain 
\begin{equation*}
 A = -\frac{2\alpha M_2M_3}{\mu_2^3}s^2 + \frac{M_3}{\mu_2}
  s-\frac{2\alpha M_3}{\mu_2^2} vs. 
\end{equation*}

	\section*{acknowledgment}
This work was supported by KAKENHI 20K20341, 22H05675, 20K03730, 
	and 23K03209. 

\end{document}